\documentclass{amsart}
\usepackage{amsfonts,amssymb,amscd,amsmath,enumerate,verbatim,calc}
\usepackage[all]{xy}

\newcommand{\CM}{Cohen-Macaulay}

\newcommand{\wrt}{with respect to}

\newcommand{\Ic}{\mathcal{I} }

\newcommand{\n}{\mathfrak{n} }
\newcommand{\m}{\mathfrak{m} }

\newcommand{\Xb}{\mathbf{X}_\bullet}
\newcommand{\Zb}{\mathbf{Z}_\bullet}
\newcommand{\Yb}{\mathbf{Y}_\bullet}
\newcommand{\Ib}{\mathbf{I}_\bullet}
\newcommand{\Pb}{\mathbf{P}_\bullet}
\newcommand{\Qb}{\mathbf{Q}_\bullet}
\newcommand{\Ub}{\mathbf{U}_\bullet}
\newcommand{\Vb}{\mathbf{V}_\bullet}
\newcommand{\Wb}{\mathbf{W}_\bullet}
\newcommand{\Fb}{\mathbf{F}_\bullet}
\newcommand{\Gb}{\mathbf{G}_\bullet}
\newcommand{\Kb}{\mathbf{K}_\bullet}
\newcommand{\Db}{\mathbf{D}_\bullet}

\newcommand{\Rc}{\mathcal{R} }
\newcommand{\Tc}{\mathcal{T} }
\newcommand{\Uc}{\mathcal{U} }

\newcommand{\Fc}{\mathcal{F} }
\newcommand{\Gc}{\mathcal{G} }
\newcommand{\Cc}{\mathcal{C} }
\newcommand{\Kc}{\mathcal{K} }
\newcommand{\Sc}{\mathcal{S} }
\newcommand{\Dc}{\mathcal{D} }
\newcommand{\Ec}{\mathcal{E} }

\newcommand{\Z}{\mathbb{Z} }
\newcommand{\pb}{\mathbf{p}}
\newcommand{\ib}{\mathbf{i}}
\newcommand{\Ga}{\Gamma}
\newcommand{\rt}{\rightarrow}

\newcommand{\rad}{\operatorname{rad}}
\newcommand{\rank}{\operatorname{rank}}

\newcommand{\depth}{\operatorname{depth}}
\newcommand{\md}{\operatorname{mod}}
\newcommand{\thick}{\operatorname{thick}}

\newcommand{\width}{\operatorname{width}}

\newcommand{\cone}{\operatorname{cone}}

\newcommand{\proj}{\operatorname{proj}}

\newcommand{\Inj}{\operatorname{Inj}}

\newcommand{\Hom}{\operatorname{Hom}}
\newcommand{\cHom}{\mathcal{H}om}
\newcommand{\Ext}{\operatorname{Ext}}

\theoremstyle{plain}

\newtheorem{theorem}{Theorem}[section]
\newtheorem{corollary}[theorem]{Corollary}
\newtheorem{lemma}[theorem]{Lemma}
\newtheorem{proposition}[theorem]{Proposition}

\newtheorem{conjecture}[theorem]{Conjecture}

\theoremstyle{definition}
\newtheorem{definition}[theorem]{Definition}
\newtheorem{s}[theorem]{}
\newtheorem{remark}[theorem]{Remark}

\theoremstyle{remark}

\numberwithin{equation}{section}
\begin{document}

\title{On a generalization of two results of Happel to commutative rings}
\author{Tony~J.~Puthenpurakal}
\date{\today}
\address{Department of Mathematics, IIT Bombay, Powai, Mumbai 400 076}

\email{tputhen@math.iitb.ac.in}
\subjclass{Primary  13D09,  16G70 ; Secondary 13H05, 13H10}
\keywords{bounded homotopy category, bounded derived category, Auslander-Reiten triangles}
 \begin{abstract}
 In this paper we extend two results of Happel to commutative rings. Let $(A,\m)$ be a commutative Noetherian local ring. Let $D^b_f(\md(A))$ be the bounded derived category of complexes of finitely generated modules over $A$ with finite length cohomology. We show $D^b_f(\md(A))$ has Auslander-Reiten(AR)-triangles if and only if $A$ is regular. Let $K^b_f(\proj A)$ be the homotopy category of finite complexes of finitely generated free $A$-modules with finite length cohomology. We show that if $A$ is complete and if $A$ is Gorenstein then $K^b_f(\proj A)$ has AR triangles. Conversely we show that if $K^b_f(\proj A)$ has AR triangles and if $A$ is \CM \ or if $\dim A = 1$ then $A$ is Gorenstein.
\end{abstract}
 \maketitle

\section{introduction}
Let $\Gamma$ be a finite dimensional algebra over a field $k$. Happel proved that its bounded derived category $D^b(\md(\Gamma))$ has Auslander-Reiten (AR) triangles if and only if $\Gamma$ has finite global dimension, see \cite[3.6]{happel-1}, \cite[1.5]{happel-2}. He also proved that $K^b(\proj \Gamma)$ the homotopy category of bounded complexes of finitely generated projective $\Gamma$-modules has right AR triangles if and only if $\Gamma$ is a Gorenstein algebra, see \cite[3.4]{happel-3}. As exemplified by Auslander that concepts in the study of representation theory of Artin algebras have natural analogues in study of maximal \CM \ modules over \CM \ local rings. In this paper we study natural analogues of Happel's results in the context of commutative Noetherian rings.

Let $(A,\m)$ be a commutative Noetherian local ring of dimension $d$.
Let \\  $C^b(\md (A))$ be the category of bounded complexes of finitely generated $A$-modules. Let   $K^b(\md (A))$ be the corresponding  homotopy category and let $D^b(\md (A)$ be the bounded derived category of $A$.
We write complexes cohomologically.
\[
\Xb \colon  \cdots \rt \Xb^i \xrightarrow{d^i} \Xb^{i+1} \xrightarrow{d^{i+1}} \Xb^{i+2}\rt \cdots
\]
with $d^{i+1} \circ d^i = 0$ for all $i$.

Let $C^b(\proj A)$   (and $C^{-,b}(\proj A)$) be   category of bounded complexes (respectively   bounded above with $H^*(\Xb)$ finitely generated for all $\Xb$) )
 of finitely generated free $A$-module and let $K^b(\proj A)$   (and $K^{-,b}(\proj A)$) be the corresponding homotopy categories.
  Recall that in general $D^b(\md A) \cong K^{-,b}(\proj A)$ and if $A$ is regular local then $D^b(\md A) \cong K^b(\md A)$.  Let $D^b_f(\md(A))$ be the bounded derived category of complexes of finitely generated modules over $A$ with finite length cohomology. Then it can be shown that $D^b_f(\md(A))$ is a Hom-finite Krull-Schmidt triangulated category, see \ref{tea}. Our generalization of Happel's results \cite[3.6]{happel-1}, \cite[1.5]{happel-2} is the following:
  \begin{theorem}\label{happel-gen}
  Let $(A,\m)$ be a commutative Noetherian local ring. The following conditions are equivalent:
  \begin{enumerate}[\rm (i)]
  \item
  $A$ is regular.
  \item
  $D^b_f(\md(A))$ has AR-triangles.
  \end{enumerate}
  \end{theorem}
Our proof is not similar to Happel's.

To state our next result we need the following notation. Let $K^b_f(\proj A)$ be the subcategory of $K^b(\proj A)$ with finite length cohomology.
By Happel's result if $(A,\m)$ is a zero-dimensional commutative Gorenstein ring then $K^b(\proj A)$ has AR-triangles. For higher dimensional Gorenstein rings we prove the following extension of one direction of Happel's result  \cite[3.4]{happel-3}.
\begin{theorem}
  \label{Gor} Let $(A,\m)$ be a complete Gorenstein local ring.
  Then $K^b_f(\proj A)$ has AR-triangles.
\end{theorem}

\begin{remark}
When $A$ is complete then Theorem \ref{happel-gen} ((i) $\implies$ (ii)) follows from Theorem \ref{Gor}. However even in the complete case it is relevant to prove Theorem\ref{happel-gen} directly. The reason is that the proof of Theorem \ref{Gor} uses a non-trivial result due to Neeman \cite{NK} while proof of Theorem\ref{happel-gen} is entirely self-contained.
\end{remark}

 We believe the converse to Theorem \ref{Gor} is true. More precisely
\begin{conjecture}
  \label{h-conj} Let $(A,\m)$ be a complete Noetherian local ring. If \\ $K^b_f(\proj A)$ has AR-triangles then $A$ is Gorenstein.
\end{conjecture}
We prove Conjecture \ref{h-conj} under the following cases:
 \begin{theorem}\label{cm-gor-intro}
Let $(A,\m)$ be a complete Noetherian local ring. Assume \\
$K^b_f(\proj A)$ has AR-triangles. Then
\begin{enumerate}[\rm (1)]
  \item if $A$ is \CM  \ then $A$ is Gorenstein.
  \item if $ \dim A = 1$ then $A$ is Gorenstein.
\end{enumerate}
 then $A$ is Gorenstein.
\end{theorem}

If $A$ is also Henselian then $C^b(\proj A)$ is also a Krull-Schmidt category. It  then can be easily verified that $K^b(\proj A)$ is also a Krull-Schmidt category.
Let $K^b_f(\proj A)$  denote the  complexes in $K^b(\proj A)$ with finite cohomology. \emph{ We give two results which justify why we are concentrating on $K^b_f(\proj A)$}.

Our first result is the following. In a Krull-Schmidt triangulated category we may consider the notion of right (left) AR( Auslander-Reiten) triangles. We prove
\begin{theorem}
  \label{ar} Let $(A,\m)$ be a Henselian commutative Noetherian local ring. Let $\Xb \in K^b(\proj A)$. If there is a right (left) AR-triangle ending at $\Xb$ then $H^*(\Xb)$ has finite length.
\end{theorem}
The techniques used to prove Theorem \ref{ar} are inspired by some techniques of Yoshino, see \cite[Chapter 2]{Y}.

To state our second result we require some preliminaries:
Let $\Xb \in C^b(\proj A)$ be a minimal complex i.e $\partial (\Xb) \subseteq \m \Xb$. Set
$$ \width \Xb = \sup \{ i \mid \Xb^i \neq 0\} - \inf \{ i \mid \Xb^i \neq 0\}$$
and
$$ \rank \Xb = \sum_{i} \rank \Xb^i. $$
In general if $\Xb \in C^b(\proj A)$ then $\Xb = \Ub \oplus \Vb$ where $\Ub$ is minimal complex and $\Vb $ is a finite direct sum of shifts of the complex $0 \rt A \xrightarrow{1} A \rt 0$. We set
$$ \width \Xb = \width \Ub \quad \text{and} \quad \rank \Xb = \rank \Ub. $$
We note that if $\Xb \in K^b(\proj A)$ then we can define width and rank of $\Xb$ to be same as considered as a complex in $C^b(\proj A)$. It can be easily seen that  this is well-defined.

Our second result justifying study of  $K^b_f(\proj A)$ is
\begin{theorem}\label{huneke-leu}
Let $\Tc$ be a thick subcategory of $K^b(\proj A)$ and let $\Xb \in \Tc$. Let $m = \width \Xb$ and $r = \rank \Xb$. Suppose $\Tc$  has only finitely many complexes (upto shift) with width $\leq 2m$ and rank $\leq 2r$. Then $H^*(\Xb)$ has finite length.
\end{theorem}
The proof of the above result is an abstraction of a proof due to Huneke and Leuschke, see \cite[Theorem 1]{HL}.
To get our abstraction we need an analogue of a theorem of Miyata \cite[Theorem 1]{Mi}.
Miyata proved that if $R$ is a commutative  Noetherian ring and let $\Ga$ be an $R$- algebra which is  finitely generated  as a $R$-module. Let $M, N, L$ be  finitely generated $\Ga$  modules such that there is an exact sequence
$$ s \colon 0 \rt M   \rt  N \rt L   \rt 0$$
If $N \cong M \oplus L$ then $s$ is split.

Let $\Cc$ be a triangulated category.
If $s \colon M \rt N \rt L \xrightarrow{0} M[1]$ is a triangle then it is well-known that  $L \cong N \oplus M$; see \cite[1.2.7]{N}.
Our analogue of Miyata's theorem is as follows:
\begin{theorem}\label{Miyata-proj}
 Let $\Ga$ be a left Noetherian ring (not necessarily commutative).  Suppose we have a triangle in $K^{-,b}(\proj \Ga)$
\[
s \colon \Ub \xrightarrow{u} \Wb \xrightarrow{w} \Vb \xrightarrow{v} \Ub[1],
\]
such that $\Wb  \cong \Ub \oplus \Vb$ then $v = 0$ (equivalently  there is $\xi \colon \Vb \rt \Wb$ such that $w \circ \xi = 1_{\Vb}))$.
\end{theorem}

 We now describe in brief the contents of this paper. In section two we discuss some preliminary results that we need. In section three we discuss some preliminaries on AR-sequences that we need. In section four we prove that
 $D^b_f(A)$ is a Krull-Schmidt category. In the next section we give a proof of Theorem \ref{happel-gen}( (i) $\implies$ (ii)). In section six we give a proof of Theorem \ref{Gor}. In section seven  we discuss a filtration of certain DG-algebras that we need. In section eight we give a proof of Theorem \ref{happel-gen}( (ii) $\implies$ (i)). In section nine  we give a proof of Theorem \ref{cm-gor-intro}. In the next section we give a proof of Theorem \ref{ar}. In section eleven we prove our analogue of Miyata's theorem. Finally in section twelve we give a proof of Theorem \ref{huneke-leu}.

\section{preliminaries}
We discuss a few preliminarly notions and results that we need.
\s \emph{(co-chain)-Complexes:}
Let $\Gamma$ be a (not-necessarily commutative) left Noetherian ring. We will consider co-chain complexes over $\Gamma$
\[
\Xb \colon  \cdots \rt \Xb^i \xrightarrow{\partial^i} \Xb^{i+1} \xrightarrow{\partial^{i+1}} \Xb^{i+2}\rt \cdots
\]
with $\partial^{i+1} \circ \partial^i = 0$ for all $i$. By $H^i(\Xb)$ we denote the $i^{th}$ cohomology of $\Xb$. By the $m^{th}$ shift  of $\Xb$ we mean the co-chain complex $\Xb[m]$ with $\Xb[m]^i = \Xb^{i+m}$ and differential $d[m]^{i} = (-1)^md^{i+m}$.

\s \emph{cones of maps:} Let $f \colon \Ub \rt \Vb$ be a co-chain map. Then by $\cone(f)$ we denote the co-chain complex with $\cone(f)^n = \Ub^{n+1} \oplus \Vb^n$ with differential given by
$$ \partial^n(u,v) = (-\partial^{n+1}(u), \partial^n(v) - f^{n+1}(u)).$$
We also have an exact sequence
\begin{equation}\label{c-dia}
  0 \rt   \Vb    \rt       \cone(f)      \rt \Ub[1] \rt 0
\end{equation}
where the left map sends $v$ to $(0,v)$ and the right map sends $(u,v)$ to $-u$.
We need the following result
\begin{lemma}\label{cone-center}
Let $r $ be in center of $\Ga$.  Let $f \colon \Ub \rt \Vb$ be a map. Then we have a  commutative diagram
\[
  \xymatrix
{
 0
 \ar@{->}[r]
  & \Vb
\ar@{->}[r]
\ar@{->}[d]^{r}
 & \cone(f)
\ar@{->}[r]
\ar@{->}[d]^{\psi}
& \Ub[1]
\ar@{->}[r]
\ar@{->}[d]^{j}
&0
\\
 0
 \ar@{->}[r]
  & \Vb
\ar@{->}[r]
 & \cone(rf)
\ar@{->}[r]
& \Ub[1]
    \ar@{->}[r]
    &0
\
 }
\]
where the horizontal maps are as in \ref{c-dia} and $j$ is the identity map on $\Ub[1]$.
\end{lemma}
\begin{proof}
Define $\psi$ as
$$ \psi(u,v) = (u,rv).$$
The commutativity of the diagram can be easily verified.
\end{proof}
\s \emph{Triangulated categories:}\\
We use \cite{N} for notation on triangulated categories. However we will assume that if $\mathcal{C}$ is a triangulated category then $\Hom_\mathcal{C}(X, Y)$ is a set for any objects $X, Y$ of $\mathcal{C}$.

\s By $C^b(\md \Gamma)$ we mean the category of bounded co-chain complexes of finitely generated left $\Gamma$-modules. Let $K^b(\Gamma)$ be the corresponding homotopy category.
Furthermore let $D^b(\md(\Gamma))$ be its derived category. We note that $K^b( \Ga)$ and $D^b(\md (\Ga))$ are triangulated categories with shift functor $[1]$. Let $C^{-,b}(\proj \Gamma)$ be the category of bounded above complexes of finitely generated projective $\Ga$-modules with finitely generated cohomology and let $K^{-,b}(\proj \Ga)$ be the corresponding homotopy category. It is well-known that we have an equivalence of triangulated categories
$ K^{-,b}(\proj \Ga) \cong  D^b(\md(\Gamma))$.

\s We now work over a \emph{commutative} Noetherian ring $A$. We first note that if $\Ub, \Vb \in K^b(A)$ then it is easy to verify that for any prime $P$ in $A$ the natural map
$$ \Hom_{K^b(A)}(\Ub, \Vb)_P \rt \Hom_{K^b(A_P)}({\Ub}_P, {\Vb}_P) $$
is an isomorphism.
As  a consequence we get
\begin{corollary}
\label{hl-lemma}Let $(A,\m)$ be a Noetherian local ring. Let $\Xb \in K^b(\proj A)$ be such that for any $u \in \Hom_{K(\proj A}(\Xb, \Xb)$ and any $r \in \m$ we have $r^n u = 0$ then $ H^*(\Xb)$ has finite length.
\end{corollary}
\begin{proof}
  Let $P \neq \m$ be a prime. Our assumptions imply that
  $$\Hom_{K(\proj A)}(\Xb, \Xb)_P = 0.$$
   By our earlier observation we get that
 $ \Hom_{K(\proj A_P)}({\Xb}_P, {\Xb}_P) = 0$. So ${\Xb}_P = 0$. It follows that $H^*({\Xb}_P) = 0$. Therefore $H^*(\Xb)_P = 0$. As $P$ is any non-maximal prime the result follows.
\end{proof}

\s \label{chom}  Let $\Ub, \Vb $ be complexes. The \emph{complex of homomorphisms} from $\Ub$ to    $\Vb$  is a complex $\cHom_A(\Ub, \Vb)$ consisting of modules
$$ \cHom_A(\Ub,  \Vb)^n = \prod_{i}\Hom_A(\Ub^i, \Vb^{i+n}), $$
and differential defined on a element
$f$ in $ \cHom_A(\Ub,  \Vb)^n $ by
$$ \partial(f) = \partial_{\Vb}\circ f   -  (-1)^nf\circ \partial_{\Ub}. $$
It is elementary to observe that
$$ H^0(\cHom_A(\Ub, \Vb)) = \Hom_{K(A)}(\Ub,  \Vb). $$

\s \label{chom-p} It is clear that if $\Ub, \Vb \in C^b(\md(A))$ then $ \cHom_A(\Ub,  \Vb) \in C^b(\md(A))$.  Furthermore if $P$ is a prime ideal in $A$ then
\begin{equation}
  \cHom_A(\Ub, \Vb)_P   \cong    \cHom_{A_P}({\Ub}_P, {\Vb}_P).
\end{equation}

\s\label{dual} Let $\Ub \in C^b(\proj A)$. Then set $\Ub^* = \cHom_A(\Ub, A)$ where $A$ is considered as a complex concentrated in degree zero. Clearly we have
\begin{enumerate}
  \item $(\Ub^*)^* \cong \Ub$.
  \item If $\Ub$ is an exact complex then so is $\Ub^*$.
  \item By (2) and \ref{chom-p} it follows that if $\ell(H^*(\Ub)) < \infty$  then $\ell(H^*(\Ub^*)) < \infty$.
\end{enumerate}

Consider $\Dc(-) = \cHom(-, A)$. Then $\Dc$ is a $A$-linear contravariant  functor from $K^b(\proj A)$ to $K^b(\proj A)$ which also restricts to a functor from
$K^b_f(\proj A)$ to $K^b_f(\proj A)$. Note $\Dc^2 \cong 1$.

\s \label{duality-p} We can consider $\Dc$ as (covariant) functor from $K^b(\proj A) \rt K^b(\proj A)^{op}$. It is routine but tedious to verify that $\Dc$ is a triangulated functor and in-fact an equivalence of triangulated categories.

\s Let $\Inj A$ denote the category of injective $A$-modules. Let $K^b(\Inj A)$ be the homotopy category of bounded complexes of injective $A$-modules. Let $E$ denote the injective hull of $k = A/\m$.  Let $K^b(E)$ be the thick category of $K^b(\Inj A)$ generated by $E$. Let $\Ec(-) = \cHom_A(-, E)$. Then $\Ec \colon K^b(\proj A) \rt K^b(E)^{op}$
is a triangulated functor.

\s \label{duality-e} Let $(A,\m)$ be a complete Noetherian local ring. Then by Matlis duality we get $\Hom_A(E, E) = A$. Using this fact it follows that for any $\Xb \in K^b(\proj A)$ and $\Qb \in K^b(E)$ we have natural isomorphism of complexes
$$ \Ec(\Ec(\Xb)) \cong \Xb \quad \text{and}    \quad \Ec(\Ec(\Qb)) \cong \Qb.             $$
So the functor $\Ec \colon K^b(\proj A) \rt K^b(E)^{op}$ is an equivalence of categories.

\s If $A$ is also Henselian then $C^b(\proj A)$ is also a Krull-Schmidt category. It then can be easily verified that $K^b(\proj A)$ is also a Krull-Schmidt category.

\section{Preliminaries on AR-triangles}

 The notion of AR-triangles in a triangulated Krull-Schmidt category was introduced by Happel \cite[Chapter 1]{Happel}.
First note that the notion of split monomorphism and split epimorphism makes sense in any triangulated category. They are called section and retraction respectively in \cite{Happel}.  However in \cite{Happel} AR-triangles are defined only in Hom-finite categories and in this case right AR-triangles are also left AR-triangles (and vice-versa). However we need the notion of AR-triangles (both left and right) in general. So in this section we carefully define the relevant notions.

Let $\Cc$  be a Krull-Schmidt triangulated category with shift functor $ \sum$.

A triangle $ N \xrightarrow{f} E \xrightarrow{g} M \xrightarrow{h} \sum N$  in $\Cc$ is called a \emph{ right AR-triangle} (ending at $M$) if

   (RAR1) \  \ $M, N$ are indecomposable.

   (RAR2) \ \  $h \neq 0$.

   (RAR3) \  \ If $D$ is indecomposable then for every non-isomorphism $t \colon D \rt M$ we have $h\circ t = 0$.

By considering the functor $\Hom_\Cc(D, -)$ it is easy to see that (AR3) is equivalent to

  (RAR3l)  \  \ If $D$ is indecomposable then for every non-isomorphism $t \colon D \rt M$    there exists $q \colon D \rt E$ with $g\circ q = t$.

It is also easy to see that (RAR3l) is equivalent to

(RAR3g)  \  \ If $W$ is not-necessarily indecomposable and  $s \colon W \rt M$ is not a retraction then    there exists $q \colon W \rt E$ with $g\circ q = s$.

Dually,  a triangle $\sum^{-1} M \xrightarrow{w} N \xrightarrow{f} E \xrightarrow{g} M $  in $\Cc$ is called a \emph{left AR-triangle} (starting at $N$) if

   (LAR1) \  \ $M, N$ are indecomposable.

   (LAR2) \ \  $w \neq 0$.

   (LAR3) \  \ If $D$ is indecomposable then for every non-isomorphism $t \colon N \rt D$ we have $t\circ w = 0$.

By considering the functor $\Hom_\Cc( - , D)$ it is easy to see that (LAR3) is equivalent to

  (LAR3l)  \  \ If $D$ is indecomposable then for every non-isomorphism $t \colon N \rt D$    there exists $q \colon E\rt D$ with $ q \circ f = t$.

It is also easy to see that (LAR3l) is equivalent to

(LAR3g)  \  \ If $W$ is not-necessarily indecomposable and  $s \colon N \rt W$ is not a section then    there exists $q \colon E \rt W$ with $q\circ f = s$.

\begin{definition}
We say a Krull-Schmidt triangulated category $\Cc$ has AR-triangles if for any indecomposable $M \in \Cc$ there exists a right AR-triangle ending at $M$ and a left AR-triangle starting at $M$.
\end{definition}

In \cite[1.4.2]{Happel} it is proved that in a Hom-finite Krull-Schmidt category if
$ N \xrightarrow{f} E \xrightarrow{g} M \xrightarrow{h} \sum N$  in $\Cc$ is a right AR-triangle ending at $M$ then its rotation
$\sum^{-1} M \xrightarrow{-\sum h} N \xrightarrow{f} E \xrightarrow{g} M $  is a left AR-triangle starting at $N$. A careful reading of the proof actually shows the following
\begin{proposition}\label{right-to-left}
Let $ N \xrightarrow{f} E \xrightarrow{g} M \xrightarrow{h} \sum N$  in $\Cc$ is a right AR-triangle ending at $M$. If $\Hom_\Cc(N, N)$ is Artinian then its rotation
$\sum^{-1} M \xrightarrow{-\sum^{-1} h} N \xrightarrow{f} E \xrightarrow{g} M $  is a left AR-triangle starting at $N$.
\end{proposition}

The following dual statement can also be proved similarly (but in our context also see \ref{our-int}).
\begin{proposition}\label{left-to-right}
Let $ \sum^{-1} M \xrightarrow{w}N \xrightarrow{f} E \xrightarrow{g} M $  in $\Cc$ is a left AR-triangle ending at $M$. If $\Hom_\Cc(M, M)$ is Artinian then its rotation
$ N \xrightarrow{f} E \xrightarrow{g} M \xrightarrow{- \sum w  } \sum N $  is a right AR-triangle ending at $M$.
\end{proposition}

 \begin{remark}\label{HuniqueAR}
 It is shown in \cite[1.4.3]{Happel} that AR triangles are unique up-to isomorphism of triangles. However the proof \emph{only works} in Hom-finite categories.
   We  prove uniqueness of AR-triangles in $\Cc$ in general.
 \end{remark}
 We first prove:
 \begin{lemma}\label{end-tri}
 Let $s \colon  N \xrightarrow{f} E \xrightarrow{g} M \xrightarrow{h} \sum N$  be a triangle in $\Cc$ with $N$ indecomposable and $h \neq 0$. Let
 $\theta \colon s \rt s$ be a morphism of triangles such that $\theta_M = 1_M$.  Then $\theta$ is an isomorphism of triangles.
 \end{lemma}
 \begin{proof}
 We note that $\sum N$ is indecomposable. If $u = \theta_{\sum N}$ is not an isomorphism then $u$ is in the Jacobson radical of $ \Rc = \Hom_{\Cc}(\sum N, \sum N)$.   As $s$ is a morphism of triangles we get $u \circ h = h$. Therefore $(1_\Rc- u)\circ h = 0$. Note $1_\Rc - u$ is invertible. This forces $h = 0$; a contradiction. So $u$ is an isomorphism. It follows that $\theta_N$ is an isomorphism. As $\theta_M, \theta_N$ are isomorphisms it is well known that $\theta_E$ is also an isomorphism. So $\theta$ is an isomorphism of triangles.
 \end{proof}
 We now prove uniqueness of right AR-triangles in $\Cc$.
 \begin{proposition}
\label{unique-AR}
In $\Cc$ consider the following two  right AR-triangles ending at $M$.
\begin{itemize}
  \item $s \colon N \xrightarrow{f} E \xrightarrow{g} M \xrightarrow{h} \sum N$
  \item $s' \colon  N' \xrightarrow{f'} E' \xrightarrow{g'} M \xrightarrow{h'} \sum N'$
\end{itemize}
Then $s \cong s'$.
 \end{proposition}
 \begin{proof}
   As $h, h'$ are non-zero we get that $g, g'$ are not retractions; see \cite[1.4]{Happel}. By (AR3g) there exists lifts $q \colon E' \rt E$ and $q' \colon E \rt E'$ lifting $g'$ and $g$ respectively. So we have morphism of triangles $\alpha \colon s \rt s'$ and $\beta \colon s' \rt s$ with
   $\alpha_M = \beta_M = 1_M$. Consider $\beta\circ \alpha \colon s \rt s$. As $N$ is indecomposable and $h \neq 0$ we get by Lemma \ref{end-tri} that $\beta \circ \alpha$ is an isomorphism. Similarly $\alpha \circ \beta$ is an isomorphism. So $s \cong s'$.
 \end{proof}

\begin{definition}
  Let $s \colon N \xrightarrow{f} E \xrightarrow{g} M \xrightarrow{h} \sum N$ be a right AR-triangle ending at $M$. Then by Proposition \ref{unique-AR} $N$ is  determined
  by $M$ (upto a not-necessarily unique  isomorphism). Set $\tau^l(M) = N$.
\end{definition}
\begin{remark}
Dually one can show that left AR-triangle starting at $N$ are isomorphic.
\end{remark}

We now concentrate on when $\Cc = K^b(\proj A)$ where $(A,\m)$ is a Henselian Noetherian local ring. We consider the behaviour under the duality operator. More precisely we prove
\begin{lemma}
  \label{ar-dual} Let $(A,\m)$ be a Henselian Noetherian local ring. In $K^b(\proj A)$ we have
  \begin{enumerate}[ \rm(1)]
    \item    Let $s \colon \Ub \rt \Kb \rt \Vb \xrightarrow{h} \Ub[1]$ be a right AR-triangle ending at $\Vb$. Consider $h^* \colon \Ub^*[-1] \rt \Vb^*$. Then a triangle
    $l \colon \Ub^*[-1] \xrightarrow{h^*} \Vb^* \rt \Wb \rt \Ub^*$ is a left AR-triangle starting at $\Vb^*$
\item Let $l \colon \Ub[-1] \xrightarrow{h} \Vb \rt \Wb \rt \Ub$ be a left AR-triangle starting at $\Vb$. Consider $h^* \colon \Vb^* \rt \Ub^*[1]$. Then a triangle
    $s \colon \Ub^* \rt \Kb \rt \Vb^* \xrightarrow{h^*} \Ub^*[1]$ be a right AR-triangle ending at $\Vb^*$.
  \end{enumerate}
  \end{lemma}
\begin{proof}
  We prove (1). The assertion (2) can be proved similarly.

(1) Clearly $\Ub^*$ and $\Vb^*$ are indecomposable. Also as $(h^*)^* = h$ we get $h^* \neq 0$. Let $\Db \in K^b(\proj A)$ be  indecomposable and let $t \colon \Vb^* \rt \Db$ is not an isomorphism. Then $t^* \colon \Db^* \rt \Vb$ is not an isomorphism. Also $\Db^*$ is indecomposable. As $s$ is an AR-sequence ending at $\Vb$ we get that $h\circ t^* = 0$. So $t\circ h^* = 0$. It follows that $l$ is an AR-sequence starting at $\Vb^*$.
\end{proof}

\s \label{our-int} As an application of Lemma \ref{ar-dual} we give a proof of Proposition \ref{left-to-right} in the case of our interest.
\begin{proof}[Proof of \ref{left-to-right} when $\Cc = K^b(\proj A)$]
Let $s$  be the given left AR-triangle. By \ref{ar-dual} we get that we have a right AR-triangle ending at $N^*$.
$$l \colon  M^* \rt E \rt     N^*\xrightarrow{w^*} M[1] $$
As $H^*(M) $ has finite length it follows that $H^*(M)$ has finite length. Thus \\ $\Hom_{K^b(\proj A)}(M^*, M^*)$ is Artinian. So by \ref{right-to-left} the rotation of $l$
$$ t \colon N^*[-1] \xrightarrow{-w^*[-1]} M^* \rt E \rt N^*,$$
is a left AR-triangle starting ar $M^*$. So again by \ref{ar-dual} we get a right AR triangle ending at $M$,
$$p \colon          N \rt W \rt                        M \xrightarrow{-w[1]} N[1]$$
is a right AR-triangle ending at $M$. As a triangle is determined upto isomorphism by any of its maps the result follows.
\end{proof}

In a fundamental work Reiten and Van den Bergh, \cite{RV}, related existence of AR-triangles to existence of a Serre functor.
Let $(A,\m)$ be a Noetherian local ring and let $E$ be the injective hull of $k = A/\m$.
Set $(-)^\vee = \Hom_A(-, E)$.
 Let $\Cc$ be a Hom-finite $A$-triangulated category.  By a right Serre-functor  on $\Cc$ we mean an additive functor $F \colon \Cc \rt \Cc$ such that we have isomorphism
 \[
 \eta_{C, D} \colon \Hom_{\Cc}(C, D) \rt \Hom_{\Cc}(D, F(C))^\vee,
 \]
 for any $C, D \in \Cc$ which are natural in $C$ and $D$. It can be shown that if $F$ is a right Serre functor
 then $F$ is fully faithful. If $F$ is also dense then we call $F$ to be a Serre functor.
 We will use the following result:
 \begin{theorem}
 \label{RV-main}[see \cite[Theorem I.2.4]{RV}]
 Let $\Cc$ be a Hom-finite $A$-linear triangulated category. Then the following are equivalent
 \begin{enumerate}[\rm (i)]
 \item
 $\Cc$ has AR-triangles.
 \item
 $\Cc$ has a Serre-functor.
 \end{enumerate}
 \end{theorem}

 \begin{remark}
 In \cite{RV}, Theorem I.2.4 is proved for $k$-linear Hom-finite triangulated categories where $k$ is a field. However the same proof generalizes to our situation.
 \end{remark}

 \section{Krull-Schmidt property of $D^b_f(A)$}
In this section we prove:
\begin{lemma}\label{tea}
Let $(A,\m)$ be a Noetherian local ring. Then $D^b_f(A)$ is a Hom-finite, triangulated Krull-Schmidt category.
\end{lemma}
\begin{proof}
It is elementary that $D^b_f(A)$ is a thick subcategory of $D^b(A)$. So in particular it is a triangulated category.

Next we prove that idempotents split in $D^b_f(A)$. Let $\Xb \in D^b_f(A)$. Let $e \colon \Xb \rt \Xb $ be an idempotent. Then as idempotents split in $D^b(A)$ (see \cite[2.10]{BS} or \cite[Corollary A]{LC}) we have that there exists $\Ub \in D^b(A)$ and maps
$\alpha \colon \Xb \rt \Ub$ and $\beta \colon \Ub \rt \Xb$ such that $\beta\circ \alpha = e$ and $\alpha \circ \beta  = 1_{\Ub}$. Taking cohomology we get an inclusion $H^i(\beta) \colon  H^i(\Ub) \rt H^i(\Xb)$ for all $i \in \Z$. It follows that $\Ub \in D^b_f(A)$. Thus idempotents split in $D^b_f(A)$.

Set $\Kc = \Kc^{-,b}(\proj A)$. It is well-known that $D^b(A) \cong \Kc$. Let $\Xb, \Yb$ be bounded complexes in $D^b_f(A)$. Let $\Fc$ be a projective resolution of $\Xb$ with $\Fc^i$ finitely generated free $A$-modules and $\Fc^i = 0$ for $i \gg 0$.
Then
we know that $$\Hom_{D^b(A)}(\Xb, \Yb) \cong \Hom_{\Kc}(\Fc, \Yb).$$

Let $\Ec = \cHom_A(\Fc, \Yb)$ be the complex of homomorphisms from $\Fc$ to $\Yb$; see \ref{chom}. Then it is clear that $\Ec^i = 0$ for all $i \gg 0$ and $\Ec^i$ is finitely generated $A$-module for all $i \in \Z$ (note $\Yb$ is a bounded complex). So $H^i(\Ec)$ is a  finitely generated $A$-module for all $i \in \Z$.
In particular $H^0(\Ec))  = \Hom_{\Kc}(\Fc, \Yb)$ is finitely generated $A$-module.
Let $P \neq \m$ be a prime ideal in $A$. Set
${\Kc}_P = \Kc^{-,b}(\proj A_P)$.

 We note that  ${\Ec}_P \cong \cHom_{A_P}({\Fc}_P, {\Yb}_P)$.
Furthermore \\  $H^0(\Ec))_P  \cong \Hom_{\Kc_P}({\Fc}_P, {\Yb}_P)$. But ${\Fc}_P$ has zero homology. As it is bounnded above we get that ${\Fc}_P = 0$ in $\Kc_P$. Thus  $H^0(\Ec))_P = 0$ for any prime $P \neq \m$. As it is also finitely generated we get that $H^0(\Ec))$  has finite length. Thus $D^b_f(A)$ is a Hom-finite category

Next we show that any object in $D^b_f(A)$ is a finite direct sum of indecomposable objects.
Let $\Xb \in D^b_f(A)$. Let $c = h(\Xb) = \sum_i H^i(\Xb)$. Then it is clear that $\Xb$ cannot be a direct sum of $c+1$ non-zero complexes in $D^b_f(A)$. Thus $\Xb$ is a finite direct sum of indecomposable objects in $D^b_f(A)$. As $D^b_f(A)$ is Hom-finite  with split-idempotents it follows that the endomorphism rings of indecomposable objects are local, see \cite[1.1.5]{LW}. It follows that $D^b_f(A)$ is a Krull-Schmidt category.
\end{proof}

   \section{Proof of Theorem \ref{happel-gen}( (i) $\implies$ (ii))}
In this section we give a proof of Theorem \ref{happel-gen}( (i) $\implies$ (ii)). We need the following preliminaries to prove our result.
\s \label{setup-h}Let $(A,\m)$ be a Noetherian local ring with residue field $k$. Let $E = E(k)$ the injective hull of $k$. Let $\Yb \in C^b(\md(A))$ and let $\Xb \in C^b(\proj A)$. Then note that the natural map
$$ \Psi_{\Yb, \Xb} \colon \Yb \otimes \cHom_A(\Xb, A) \rt \cHom_A(\Xb, \Yb),$$
defined by
$$\Psi_{\Yb, \Xb}(y\otimes f)(x) =  f(x)y$$
is an isomorphism of complexes.

\s \label{dual-E}
Set
$$\Dc(-) =  \cHom_A(-, A) \quad \text{and} \quad \Ec(-) = \cHom_A(-, E).$$
By \ref{setup-h} we have that for any $\Yb \in C^b(\md(A))$ and  $\Xb \in C^b(\proj A)$ we have a natural isomorphism
\begin{align*}
  \Ec (\cHom_A(\Xb, \Yb)) &= \cHom_A(\cHom_A(\Xb, \Yb), E), \\
  &\cong  \cHom_A(\Yb\otimes \Dc(\Xb), E) \\
   &\cong \cHom_A(\Yb, \Ec \circ \Dc(\Xb) ).
\end{align*}

\s Instead of working with $C^b(\md(A))$ it is bit convenient to work with $C^{-,fg}(A)$ consisting of bounded complexes with finitely generated cohomology (the modules in the complexes need not be finitely generated. Let $K^{-,fg}(A)$ be it's homotopy category. After inverting all quasi-isomorphisms we obtain $D^{-,fg}(A)$. It is easy to show that $D^{-,fg}(A) \cong D^b(\md(A))$. Also recall the equivalence from $D^{-,fg}(A)$ to $K^{-,b}(\proj A)$  ($K^{+,b}(\Inj A)$)is given by the projective ( injective) resolution functor $\pb$ ($\ib$).

\s\label{reg-hensel} Now assume that $(A,\m)$ be a Henselian regular local ring. Consider the functor $F$ which is the following composite of $A$-linear functors:
$$ K^b_f(\proj A) \xrightarrow{\Dc} K^b_f(\proj A) \xrightarrow{\Ec} K^b_f(\Inj A) \xrightarrow{\pb} K^b_f(\proj A).$$
\begin{theorem}
  \label{serre} (with hypotheses as in \ref{reg-hensel}) The functor $F$ is dense.
\end{theorem}
We give a proof of Theorem \ref{happel-gen} assuming Theorem \ref{serre}.
\begin{proof}[ Proof of Theorem \ref{happel-gen}]
Set $(-)^\vee = \Hom_A(-, E)$.
Then we have for $\Xb, \Yb \in \Kc = K^b_f(\proj A)
$\begin{align*}
   \Hom_{\Kc}(\Xb, \Yb)^\vee &= (H^0(\cHom_A(\Xb, \Yb)))^\vee \\
    &\cong H^0(\Ec(\cHom_A(\Xb, \Yb))) \\
    &\cong  H^0(\cHom_A(\Yb, (\Ec\circ\Dc)(\Xb))) \\
    &= \Hom_{K^{-,fg}(A)}(\Yb, (\Ec\circ\Dc)(\Xb)) \\
    &\cong \Hom_{D^{-,fg}(A)}(\Yb, (\Ec\circ\Dc)(\Xb)) \\
    &= \Hom_{D^{-,fg}(A)}(\Yb, \pb((\Ec\circ\Dc)(\Xb)))  \\
    &\cong \Hom_{\Kc}(\Yb, F(\Xb)).
 \end{align*}
 Taking dual's \wrt \ $E$ again we get (as the $A$-modules considered have finite length) an isomorphism
 $$\eta_{\Xb,  \Yb} \colon \Hom_{\Kc}(\Xb, \Yb) \rt \Hom_{\Kc}(\Yb, F(\Xb))^\vee.$$
 It can be easily seen that $\eta_{\Xb,  \Yb} $ is natural in  $\Xb$ and   $\Yb$. So $F$ is a right Serre-functor on $ \Kc$. By \ref{serre} $F$ is also dense. Thus $F$ is a Serre-functor on $\Kc$. By \ref{RV-main} we get that $\Kc$ has AR-triangles.
\end{proof}
It remains to prove Theorem \ref{serre}. We need a few preliminaries to prove this result. The following result is essentially known. We give a proof due to lack of a reference.
\begin{lemma}\label{fin-inj}
(with hypotheses as in \ref{reg-hensel}) Let $\Xb \in K^b_f(A)$. Then there is a injective resolution $\Ib$ of $\Xb$ such that $\Ib$ is bounded and for all $n \in \Z$ we have $\Ib^n = E^{\alpha_n}$ for some $\alpha_n < \infty$.
\end{lemma}
\begin{proof}
  We induct on $\ell(H^*(\Xb))$.
  We first consider the case when $\ell(H^*(\Xb)) = 1$. We may assume that $\Xb$ is minimal and $\Xb^j = 0$ for $j > 0$. It follows that $H^0(\Xb) = k$ and $H^i(\Xb) = 0$ for $i< 0$. Thus $\Xb$ is a projective resolution of $k$. A minimal injective resolution of $k$ is the required injective resolution of $\Xb$.

  Let $\ell(H^*(\Xb)) =  n$ and assume the result is proved for all complexes $\Yb$ with $\ell(H^*(\Yb)) < n$. We may assume that $\Xb$ is minimal and $\Xb^j = 0$ for $j > 0$.
  It follows that $W = H^0(\Xb) \neq 0$. If $H^i(\Xb) = 0$ for $i \neq 0 $ then $\Xb$  is a projective resolution of $W$. A minimal injective resolution of $W$ yields the required injective resolution of $\Xb$.
  Next consider the case when $H^i(\Xb)\neq  0$ for some $i \neq 0 $.
  Let $\Pb$ be a minimal projective resolution of $W$. So we have a co-chain map $\psi \colon \Xb   \rt \Pb$ which induces an isomorphism on $H^0(-)$. It follows that $H^*(\cone(\psi)) $ has length $< n$. We have a triangle in $K^b(\proj A) $
  $$ \Xb \rt \Pb \rt \cone(\psi) \rt \Xb[1].$$  We have an equivalence $K^b(\proj A) \rt K^{b,fg}(\Inj A)$ obtained by taking injective resolutions. So we have a triangle
  $$ \ib(\Xb)    \rt   \ib(\Pb)   \xrightarrow{g}      \ib(\cone(\psi))   \rt  \ib(\Xb)[1]. $$
  By our induction hypotheses we can choose $\ib(\Pb), \ib(\cone(\psi))$ of the required form. Now notice that
  $$ \ib(\Xb)[1]  \cong \cone(g).$$
  The result follows.
\end{proof}
Next we show
\begin{proposition}
\label{tensor}
Let $(A,\m)$ be a Noetherian local ring. Let $\widehat{A}$ be the $\m$-adic completion of $A$.  Consider the map $\psi \colon K_f^b(\proj A) \rt K^b_f(\proj \widehat{A} )$ given
by $-\otimes \widehat{A}$.
Then
\begin{enumerate}[ \rm (1)]
  \item $\psi$ is a triangulated functor.
  \item $\psi$ is full and faithful.
  \item If $A$ is regular local then $\psi$ is also dense.
\end{enumerate}
\end{proposition}
\begin{proof}
  (1) This follows from the fact that $\widehat{A}$ is a flat $A$-algebra

  (2) For any $\Xb, \Yb$ the $A$-module $\Hom_{K_f^b(\proj A)} (\Xb, \Yb)$ has finite length. So we have
  \begin{align*}
    \Hom_{K_f^b(\proj A)} (\Xb, \Yb) &= \Hom_{K_f^b(\proj A)} (\Xb, \Yb)\otimes \widehat{A} \\
     &\cong \Hom_{K_f^b(\proj \widehat{A})} (\Xb \otimes \widehat{A}, \Yb \otimes \widehat{A}).
  \end{align*}
  The result follows.

  (3) We first note that if $W$ is any finite length $\widehat{A}$-module then it is also a finite length $A$-module.

  Let $\Xb \in K_f^b(\proj \widehat{A})$. We want to show there is $\Yb \in K_f^b(\proj A)$ with $\Xb   \cong    \Yb  \otimes \widehat{A}$.
   We induct on $\ell(H^*(\Xb))$.
  We first consider the case when $\ell(H^*(\Xb)) = 1$. We may assume that $\Xb$ is minimal and $\Xb^j = 0$ for $j > 0$. It follows that $H^0(\Xb) = k$ and $H^i(\Xb) = 0$ for $i< 0$. Thus $\Xb$ is a projective resolution of $k$. Let $\Yb$ be minimal projective resolution of $k$ as a $A$-module. Then $\Xb   \cong    \Yb  \otimes \widehat{A}$.

  Let $\ell(H^*(\Xb)) =  n$ and assume the result is proved for all complexes $\Ub$ with $\ell(H^*(\Ub)) < n$. We may assume that $\Xb$ is minimal and $\Xb^j = 0$ for $j > 0$.
  It follows that $W = H^0(\Xb) \neq 0$.

   If $H^i(\Xb) = 0$ for $i \neq 0 $ then $\Xb$  is a projective resolution of $W$. Let $\Yb$ be minimal projective resolution of $W$ as a $A$-module. Then $\Xb   \cong    \Yb  \otimes \widehat{A}$.

  Next consider the case when $H^i(\Xb)\neq  0$ for some $i \neq 0 $.
  Let $\Pb$ be a minimal projective resolution of $W$ as an $\widehat{A}$-module. As discussed before  $\Pb   \cong    \Qb  \otimes \widehat{A}$ where $ \Qb$ is minimal projective resolution of $W$ as an $A$-module.

   Note we have a co-chain map $\psi \colon \Xb   \rt \Pb$ which induces an isomorphism on $H^0(-)$. It follows that $H^*(\cone(\psi)) $ has length $< n$. We have a triangle in $K^b(\proj A) $
  $$ \Xb \rt \Pb \xrightarrow{f} \cone(\psi) \rt   \Xb[1].$$
  By our induction hypothesis $\cone(\psi) \cong    \Ub  \otimes \widehat{A}$ for some complex $\Ub \in K_f^b(\proj A)$.
  By (2) we get $f = g\otimes \widehat{A}$ for some $g \in \Hom_{K^b_f(\proj A)}(\Qb, \Ub)$. It follows that
  $$ \Xb[1] \cong \cone(g)\otimes \widehat{A}. $$
\end{proof}
We now give
\begin{proof}
  [Proof of Theorem \ref{serre}]
  Let $\Ub \in K^b_f(\proj A)$. Then by \ref{fin-inj} there is a injective resolution $\Ib$ of $\Xb$ such that $\Ib$ is bounded and for all $n \in \Z$ we have $\Ib^n = E^{\alpha_n}$ for some $\alpha_n < \infty$. Note a projective resolution of $\Ib$ is given by $\Ub$.  Then $\Qb = \Ec(\Ib) \in K^b_f(\proj \widehat{A})$. By \ref{tensor}
  we get that $\Qb \cong  \Vb\otimes \widehat{A}$ for some $\Vb \in K^b_f(\proj A)$. Note $\Ec(\Vb) \cong \Ib$. Set $\Xb = \Dc(\Vb)$. It is easy to check that
  $$ F(\Xb) \cong \Ub.$$
  So $F$ is dense.
\end{proof}

\section{Proof of Theorem \ref{Gor}}
In this section we prove Theorem \ref{Gor}. We need some preliminaries regarding thick subcategories of a triangulated categories and its behaviour under triangulated functors.

\s Let $\Tc$ be a triangulated category with shift functor $\sum$. Let $\Ic_1, \Ic_2$ be two subcategories of $\Tc$. We denote by $\Ic_1 * \Ic_2$ the full subcategory  of $\Tc$ consisting of objects $M$  such that there is a  triangle  $M_1 \rt M \rt M_2 \rt \sum M_1$ with $M_i \in \Ic_i$.
Let $\Ic$ be a full subcategory of $\Tc$. By $\langle \Ic  \rangle$ we denote the smallest subcategory of $\Tc$ containing $\Ic$  which is closed under finite direct sums, direct summands, shifts and isomorphisms. Set  $\Ic_1 \diamond \Ic_2 = \langle  \Ic_1 * \Ic_2 \rangle$.

Set $\langle \Ic \rangle_0 = 0$. Then inductively define $\langle   \Ic   \rangle_i = \langle \Ic \rangle_{i-1} \diamond  \langle \Ic  \rangle.$
It is clear that $\langle \Ic \rangle_{i-1} $ is a subcategory of  $\langle   \Ic   \rangle_i$ for all $i \geq 1$.
Set
$$ \langle \Ic\rangle_{\infty} = \bigcup_{i \geq 0} \langle   \Ic   \rangle_i.$$

\s \label{thick} Let $\Ic$ be a subcategory of $\Tc$. By $\thick(\Ic)$ we mean the intersection of all triangulated subcategories of $\Tc$ containing $\Ic$. It can be shown that
$ \thick(\Ic) =  \langle \Ic\rangle_{\infty}$.

\ \label{thick-functor} Let $\Tc, \Uc$ be a triangulated categories and let $F \colon \Tc \rt \Uc$ be a triangulated functor. Let $\Ic$ be a subcategory of $\Tc$. By $F(\Ic)$ we mean the full subcategory of $\Uc$ consisting of objects $F(M)$ where $M$ is an object of $\Ic$. It is clear from \ref{thick} that $F$ restricts to a triangulated functor
$F_r\colon \thick(\Ic) \rt \thick(F(\Ic))$. If $F$ is an equivalence with quasi-inverse $G$, then note that as $GF(M) \cong M $ for all $M$ we get that $G$ induces a triangulated functor
$G_r \colon \thick(F(\Ic)) \rt \thick(\Ic)$. It follows that $F_r$ is an equivalence with quasi-inverse $G_r$.

\s \label{setup-gor} For the rest of this section we assume $(A,\m)$ is a complete Gorenstein local ring of dimension $d$.  We will use notation introduced in \ref{duality-p} and \ref{duality-e}.
Let $\Sc$ be the category of finite length $A$-modules which also have finite projective dimension. As $A$  is Gorenstein each element $M$ in $\Sc$ also has finite injective dimension. If $M \in \Sc$ then $M^\vee = \Hom_A(M, E) $ has finite injective dimension. As $A$ is Gorenstein it follows that $M^\vee \in \Sc$. Thus we have a duality $(-)^\vee \colon \Sc \rt \Sc^{op}$. Also  note that $M^\vee \cong \Ext^{d}_A(M, A)$.

For each $M \in \Sc$ fix a minimal projective resolution $\Pb^M$. Set \\  $\Ib^{M} = \cHom_A(\Pb^{M^{\vee}}, E)$ which is a minimal injective resolution of $M$. Set
$$ \Fc_f = \thick(\{\Pb^M \mid M \in \Sc\}) \quad \text{in} \ K^b(\proj A) \text{and} $$
$$ \Ic_f = \thick(\{ \Ib^M \mid M \in \Sc \}) \quad \text{in} \ K^{b, fg}(\Inj A).$$
It is easily verified that $\Fc_f \subseteq K_f^b(\proj A)$ and $\Ic_f \subseteq K^b_f(E)$.

\begin{lemma} \label{basic-gor}
(with hypotheses as in \ref{setup-gor})
Consider the three equivalences \\ $\Dc \colon K^b(\proj A) \rt K^b(\proj A)^{op} $, $\Ec \colon K^b(\proj A) \rt K^b(E)^{op}$ and \\
$\pb \colon K^{-,fg}(\Inj A) \rt K^{-, fg}(\proj A)$. Then
\begin{enumerate}[ \rm (1)]
  \item $\Dc$ induces an equivalence $\Dc_r \colon \Fc_f \rt \Fc_f^{op}$.
  \item $\Ec$ induces an equivalence $\Ec_r \colon \Fc_f \rt \Ic_f^{op}$.
  \item $\pb$ induces an equivalence $\pb_r \colon \Ic_f \rt \Fc_f$.
\end{enumerate}
\end{lemma}
\begin{proof}
Let $M \in \Sc$. \\
(1) Note that $\Dc(\Pb^M) \cong \Pb^{M^{\vee}}[-d]$. The result follows from \ref{thick-functor}.

(2) Note that $\Ec(\Pb^M) = \Ib^{M^\vee}$. The result follows from \ref{thick-functor}.

(3) It is tautological that $\pb(\Ib^M) \cong \Pb^M$. The result follows from \ref{thick-functor}.
\end{proof}

\s \label{serre-gor} Consider $G \colon \Fc_f \rt \Fc_f$ which is the composite of triangle equivalences
$$ \Fc_f \xrightarrow{\Dc_r} \Fc_f^{op} \xrightarrow{\Ec_r^{op}} \Ic_f \xrightarrow{\pb_r} \Fc_f. $$

\begin{remark}
By Neeman's classification of thick subcategories of $K^b(\proj A)$, \\  see \cite[Lemma 1.2]{NK},  it follows that $K^b_f(\proj A)$ does not have proper thick subcategories. So $\Fc_f = K^b_f(\proj A)$. We had to go through the whole strategy of the above proof because \ref{serre-gor} implies the following result of which I do not know a direct proof:
\begin{proposition}
Let $\Xb \in K^b_f(\proj A)$ and $\Yb \in K^b_f(E)$. Then there is an injective resolution $\ib(\Xb) \in K^b_f(E)$ and a projective resolution $\pb(\Yb) \in K^b_f(\proj A)$.
\end{proposition}
\end{remark}
We now give
\begin{proof}
  [Proof of Theorem\ref{Gor}]
  Just as in proof of Theorem \ref{happel-gen} we have for \\ $\Xb, \Yb \in K^b_f(\proj A)$ an isomorphism
  $$ \eta_{\Xb, \Yb} \colon \Hom_{\Fc_f}(\Xb, \Yb) \rt \Hom_{\Fc}(\Yb, G(\Xb))^\vee.$$
  It is easily verified that $\eta_{\Xb, \Yb}$ is natural in $\Xb\, \Yb$. So $G$ is a right Serre functor. By \ref{serre-gor} $G$ is an equivalence. In particular $G$ is dense.
  So $G$ is a Serre-functor. So by \ref{RV-main},  $K^b_f(\proj A)$ has AR-triangles.
\end{proof}

\section{A filtration of certain DG-algebras}
We will need to recall a few basic results.
Let $A$ be a Noetherian ring. We will recall Tate's construction of DG-resolution of $A/I$ where $I$ is an ideal in $A$, see \cite{T}.

\s An associative  algebra $\Xb$  over $A$ is called an \textit{non-positive DG-algebra }over $A$
if the following hypotheses are satisfied:
\begin{enumerate}
\item
$\Xb$ is non-positively graded $\Xb =  \bigoplus_{n \leq 0} \Xb^n$ with each $\Xb^i$ a finitely generated $A$-module and $\Xb^i \Xb^j \subseteq \Xb^{i+j}$ for all $i, j \leq 0$.
\item
$\Xb$ has a unit element  $1 \in \Xb^0$ such that $\Xb^0 = A1$.
\item
$\Xb$ is strictly skew-commutative;
(for homogeneous element $x \in \Xb^i$ set $|x| = i$)
For homogeneous elements $x,y$ we have
\begin{enumerate}[\rm (i)]
\item
$ x.y = (-1)^{|x||y|}yx. $
\item
$x^2 = 0$ if $|x|$ is odd.
\end{enumerate}
\item
There exists a skew derivation $d \colon \Xb \rt \Xb$ such that
\begin{enumerate}[\rm (i)]
\item
$d(\Xb^n) \subseteq \Xb ^{n+1} $ for all $n \leq 0$.
\item
$d^2 = 0$.
\item
For $x,y$ homogeneous,
$$d(xy) = d(x)y + (-1)^{|x|}x d(y).$$
\end{enumerate}
\end{enumerate}
The map $d$ is called skew derivation on $\Xb$.

\s \label{Tate}
Next we recall Tate's process of killing cycles; \cite[section 2]{T}.
Let $\Xb$ be a non-positive DG-algebra. Let $\rho < 0$ be a negative integer. Let $t \in Z^{\rho+1}(\Xb)$ be a cycle
of degree $\rho + 1$.

Case 1: $\rho$ is odd. Let $\Yb = A \oplus AT$ be the exterior power algebra with degree of $A$ is zero and $\deg T = \rho$ and $T^2 = 0$

Case 2: $\rho$ is even. Let $\Yb = \bigoplus_{j \leq 0} AT_j$ with $AT_j \cong  A$ in degree $- j\rho$  (with $T_0 = 1$) and multiplication defined by
$$T_iT_j = \frac{(|i|+|j|)!}{|i|!|j|!}T_{i+j}.   $$
In both cases $\Xb \otimes \Yb$ is again a non-positive DG-algebra. Furthermore  there is a unique skew-derivation
$d$ on $\Yb$ extending that on $\Xb = \Xb\otimes 1$ with $d(T) = t$ in case (1) and
$d(T_j) = tT_{j+1}$ for all $j$ in case (2). Furthermore we get $H^q(\Xb\otimes \Yb) = H^q(A)$ for $q > \rho + 1$ and
$H^{\rho+1}(\Xb \otimes \Yb) = H^{\rho+1}(\Xb)/At$.

\begin{remark}
Following Tate we denote in both cases  $\Xb \otimes \Yb = \Xb<T>$ with $d(T) = t$.
\end{remark}

\s \label{setup-dg}Let $\Xb$ be a non-positive DG-algebra over $A$. We assume $\Xb^i$ is a finitely generated free $A$-module for all $i \leq 0$ and $\ell(H^*(\Xb)^n) < \infty$ for all $n \in \Z$. By a good filtration $\Fc = \{\Fb(i)\}_{i \geq 0}$ of $\Xb$ we mean
\begin{enumerate}
\item
$\Fb(i)$ is a sub-complex of $\Xb$ with  $\Fb(i)^n$  a direct summand of $\Xb^n$ for all $n \leq 0$.
\item
$\Fb(i) \subseteq \Fb(i+1)$  for all $i \geq 0$ and $\bigcup_{i \geq 0} \Fb(i) = \Xb$.
\item
$\Fb(i)^n$ is a direct summand of $\Fb(i+1)^n$ for all $n \leq 0$.

\item
There exists $c$ depending on $\Xb $ and $\Fc$ such that
$$ \Fb(i)^n  \Fb(j)^m    \subseteq  \Fb(i+j +c)^{n+m}.$$
\item
$\ell(H^*(\Fb(i))) < \infty$ for all
$i\geq 1$.
\end{enumerate}
We call $c$ the \emph{parameter} of $\Fc$.
The main result of this section
is
\begin{theorem}\label{filt-ext}
(with hypotheses as in \ref{setup-dg}.)    Let $\rho < 0 $ be such that $H^i(\Xb) = 0$ for $0 > i \geq \rho + 2$.  Suppose $H^{\rho+1}(\Xb) \neq 0$ and let  $t \in Z^{\rho +1}(\Xb)$ be a cycle
of degree $\rho + 1$ such that its residue class in  $H^{\rho +1}(\Xb)$ is non-zero. Let $\Zb = \Xb<T>$ with $d(T) = t$.
Then $\Zb$ also has a good filtration as discussed in \ref{setup-dg}.
\end{theorem}
\begin{proof}
 Let $\Fc$ be a good filtration of $\Xb$ with parameter $c$. Suppose $t \in \Fb(r)^{\rho + 1}$.
We have to consider two cases.

Case (1): $\rho$ is odd.
Note ${\Zb}^n = {\Xb}^n \oplus \Xb^{n - \rho} T$.
 Set $\Gc = \{\Gb(i)\}_{i \geq 0}$
where
\[
\Gb(i)^n = \Fb(i+r + c)^n \oplus \Fb(i)^{n-\rho}T
\]
We first show that $\Gb(i)$ is a sub-complex of $\Zb$. Let $g^n = (f(i+r+c)^n , f(i)^{n-\rho}T) \in \Gb(i)^n$.
Then note that
\[
d(g^n) = d(f(i+r+c)^n) + (-1)^{n-\rho}f(i)^{n-\rho}t  + d(f(i)^{n-\rho}T.
\]
We note that
\[
f(i)^{n-\rho}t  \in \Fb(i)^{n-\rho} \Fb(r)^{\rho + 1} \subseteq \Fb(i+r + c)^{n+1}.
\]
Thus $d(\Gb(i)^n ) \subseteq \Gb(i)^{n + 1}$. It follows that $\Gb(i)$ is a sub-complex of $\Zb$.
Clearly hypotheses $(1), (2), (3)$ is satisfied.
We note that
\begin{align*}
\Gb(i)^n \Gb(j)^m &= (\Fb(i+r + c)^n \oplus \Fb(i)^{n-\rho}T) \cdot (\Fb(j+r + c)^m \oplus \Fb(j)^{m-\rho}T) \\
                  &\subseteq  \Fb(i+j + 2r + 3c)^{n+m} \oplus \Fb(i + j + r + 2c)^{n + m-\rho}T \\
                  &= \Gb(i + j  + r + 2c)_{n+m}.
\end{align*}
So $\Gc$ has parameter $r + 2c$.
We have a short exact sequence of complexes
\[
0 \rt \Fb(i+r+c) \xrightarrow{\alpha} G(i) \xrightarrow{\beta} \Fb(i)[-\rho] \rt 0,
\]
where $\alpha$ is the usual inclusion and
\[
\beta_n(*, x_{n-\rho}T)  = (-1)^n x_{n-\rho}. \tag{*}
\]
Note as $\rho$ is odd we get that $\beta$ is a chain map.The face that $(*)$ is a short exact sequence of complexes is clear. As $\ell(H^*(\Fb(i+r+c))) $ and $\ell(H^*(\Fb(i))) $ are finite it follows that
$\ell(H^*(\Gb(i)))$ is also finite. Thus hypotheses $(4)$ of \ref{setup-dg} is satisfied.
So $\Gc$ is a good filtration of $\Zb$.

Case (2): $\rho$ is even.
We note that
\[
\Zb^n = \Xb^n \oplus  \Xb^{n - \rho}T_{-1} \oplus  \Xb^{n - 2\rho}T_{-2} \oplus \cdots
\]
Set $\Gc = \{ \Gb(i) \}_{i \geq 0}$ where
\[
\Gb(i)^n = \bigoplus_{l = 0}^{-i} \Fb( (i+l)(r + c) )^{n + \rho l} T_l
\]
We first show that $\Gb(i)$ is a sub-complex of $\Zb$.
Let $l$ with $-i \leq l \leq 0$ be fixed. Let $u = aT_l \in \Gb(i)^n $. Then
\[
d(u) = d(a)T_l + (-1)^{|a|} a t T_{l+1}.
\]
We note that $ d(a)T_l \in \Gb(i)^{n+1}$. Furthermore
\begin{align*}
atT_{l+1} &\in \Fb((j+l)(r+c))^{n+\rho l} \Fb(r)^{\rho + 1}T_{l+1},\\
&\subseteq \Fb((j+l)(r+c) + r + c)^{n+\rho l + \rho  + 1}T_{l+1}\\
&=  \Fb((j+l + 1)(r+c) )^{n + 1 +\rho (l + 1)}T_{l+1}, \\
 &\subseteq  \Gb(i)^{n+1}.
\end{align*}
Thus $\Gb(i)$ is a sub-complex of $\Zb$.
Clearly hypotheses $(1), (2), (3)$ is satisfied.
Let $u = aT_{l_1} \in \Fb((j+l_1)(r+c))^{n +\rho l_1}T_{l_1}$ and $v = bT_{l_1} \in \Fb((i+l_2)(r+c))^{m +\rho l_2}T_{l_2}$.
Then
\begin{align*}
uv &\in  \Fb((j+l_1)(r+c))^{n + \rho l_1}  \Fb((i+l_2)(r+c))^{m + \rho l_2}T_{l_1 + l_2}, \\
&\subseteq \Fb((j+l_1 + i + l_2)(r+c) + c )^{n + \rho l_1 + m +\rho l_2}T_{l_1 + l_2},\\
&\subseteq \Gb(i + j +1)^{n+m}.
\end{align*}
So $\Gc$ has parameter $1$.

Next we show that $\ell(H^*(\Gb(i))) < \infty $ for all $i \geq 0$.
Fix $i \geq 0$. Set for $s = 0, \ldots, i$,
\[
\Kb(i,s)^n =  \bigoplus_{l = 0}^{-s}\Fb( (i+l)(r + c) )^{n + \rho l} T_l
\]
It is straight forward  that $\Kb(i,s)$ is a sub-complex of $\Gb(i)$ and  $\Kb(i,i) = \Gb(i)$.
We prove by induction on $s$ that $H^*(\Kb(i,s))$ has finite length for all $s$.
For $s = 0$ we have
$\Kb(i, 0) = \Fb( (i)(r + c) )$ and so by our hypotheses $\ell(H^*(\Kb(i, 0)))$ has finite length.
Let $s \geq 1$. Note we have an obvious short exact sequence of complexes
\[
0 \rt \Kb(i, s-1) \rt \Kb(i, s) \rt \Fb((i-s)(r+c))[-\rho s] \rt 0
\]
Note by our hypotheses  $\ell(H^*(\Fb((i-s)(r+c)))$ has finite length. Furthermore by induction hypotheses  $H^*(\Kb(i,s-1))$ has finite length. Thus by induction  $H^*(\Kb(i,s))$ has finite length for $s = 0, \ldots, i$.
In particular for $i = s$ we get that $H^*(\Gb(i))$ has finite length. Thus $\Gc$ is a good filtration of $\Zb$.
\end{proof}

\begin{remark}\label{step-dg}
Let $I$ be an ideal in $A$ such that $A/I$ has finite length. Let $\Xb$ be the Koszul complex on a set of generators of $I$. Then $H^*(\Xb)$ has finite length. Then if $\Yb$ is any step in the Tate resolution of $A/I$ begining at $\Xb$ then $\Yb$ has a good filtration.
\end{remark}
\section{Proof of Theorem \ref{happel-gen}( (ii) $\implies$ (i))}
In this section we prove:
\begin{theorem}
\label{converse-happel}
Let $(A,\m)$ be a Noetherian local ring. If $D^b_f(A)$ has AR-triangles then $A$ is regular local.
\end{theorem}
Throughout we identify $D^b_f(A)$ with $K^{-,b}_f(\proj A)$ the homotopy category of bounded above complexes of finitely generated free $A$-modules with finite length cohomology.
Let $\Xb$ be a minimal projective resolution of $k$. Then clearly $\Xb$ is indecomposable in $K^{-,b}_f(\proj A)$.
We will use a result independently proved by  Gulliksen \cite{G} and Schoeller \cite{S}, i.e., one  can use the Tate process to yield a minimal resolution of $k$.
Theorem \ref{converse-happel} follows from the following result:
\begin{lemma}
\label{res-field}
Let $\Xb$ be a minimal resolution of $k$. If $A$ is not regular then there does not exist a right  AR-triangle in
$K^{-,b}_f(\proj A)$ ending at $\Xb$.
\end{lemma}
We need a few preliminaries to prove Theorem \ref{res-field}. First we show
\begin{proposition}\label{retract}
Let $(A,\m)$ be a Noetherian local ring. Let $\Ub  \in K^{-,b}_f(\proj A)$ be a minimal complex, i.e., $\partial(\Ub) \subseteq \m \Ub$. Let $\Vb \subseteq \Ub$ be a sub-complex such that $\Vb^n$ is a direct summand of $\Ub^n$ for all $n \in \Z$. If $\Vb \neq \Ub$ then the inclusion $i \colon \Vb \rt \Ub$ is not a retraction
in $ K^{-,b}_f(\proj A)$.
\end{proposition}
\begin{proof}
Suppose if possible $i$ is a retraction. Then there is $r \colon \Ub \rt \Vb$ such that $i\circ r = 1_{\Ub}$ in
$K^{-,b}_f(\proj A)$. As $\Ub$ is a minimal complex we get that  $i\circ r $ is an isomorphism on $\Ub$ (considered in $C^{-,b}_f(\proj A)$). In particular $r$ is a monomorphism in $C^{-,b}_f(\proj A)$.  So $\rank \Ub^n \leq \rank \Vb^n$ for all $n \in \Z$. As $\Vb^n$ is a direct summand of $\Ub^n$ for all $n \in \Z$ we get that $\Ub^n = \Vb^n$ for all
$n \in \Z$. Thus $\Vb = \Ub$ which is a contradiction.
\end{proof}
Next we prove the following result:
\begin{proposition}\label{extend-homotopy}
Let $(A,\m)$ be a Noetherian local ring. Let $\Ub  \in K^{-,b}(\proj A)$. Let $\Vb \subseteq \Ub$ be a sub-complex such that $\Vb^n$ is a direct summand of $\Ub^n$ for all $n \in \Z$. Let $\Wb \in K^{-,b}(\proj A)$ be such that $H^n(\Wb) = 0$ for $n \leq  m$. Suppose
$\Vb^n = \Ub^n$ for all  $n \geq  m$. Let $g \colon \Ub \rt \Wb$ be a chain map such that $g$ restricted to $\Vb$ is null-homotopic. Then $g$ is null-homotopic.
\end{proposition}
\begin{proof}
Suppose $s$ is a null-homotopy of the map $\Vb \rt \Wb$. We extend it to $ \widetilde{s} \colon \Ub \rt \Wb$ as follows. On $\Vb$ define  $ \widetilde{s}$ to equal $s$. As $\Vb^n = \Ub^n$ for all  $n \geq  m$ we have defined
$ \widetilde{s}$ on $\Ub^n$ for $n \geq m$.

Suppose $ \widetilde{s}$ is defined on $\Ub^n$ for all $n \geq r +1 (\geq m)$ then we define  $ \widetilde{s}$
on $\Ub^r$ as follows: We note that $\Ub^r = \Vb^r \oplus F$ for a finitely generated free $A$-module $F$. Let
$e_1, \ldots e_l$ be a basis of $F$.
We note that
\[
\xi = g(e_j) - \partial(\widetilde{s}(\partial(e_j))  \quad  \text{is a cycle.}
\]
As $H^i(\Wb) = 0 $ for $i < m$ and $r < m$ we get that $\xi $ is a boundary. Choose $\xi^\prime $ such that
$\partial(\xi^\prime) = \xi$. Define $\widetilde{s}(e_j) = \xi^\prime$. Thus we have defined  $\widetilde{s}$  on $\Ub^r$. Inductively we have defined $\widetilde{s}$  on $\Ub$. Thus $g$ is null-homotopic.
\end{proof}
We now give:
\begin{proof}[Proof of Theorem \ref{res-field}]
Suppose if possible there exists a right AR-triangle ending at $\Xb$.
Say we have an AR-triangle
$$ \alpha \colon \Ub \rt \Kb \rt \Xb \xrightarrow{g} \Ub[1]$$
in $K^{-,b}_f(\proj A)$. Assume $H^i(\Ub[1]) = 0 $ for $i  \geq - m$.
We consider two cases:

Case (1): $A$ is a complete intersection. Let ${\Yb}_1 $ be the Koszul complex on a minimal set of generators of $\m$.
Let ${\Yb}_2 = {\Yb}_1<T_1,\ldots T_r>$ be the DG-complex obtained by killing cycles in degree minus one of ${\Yb}_1$. Then Tate shows that ${\Yb}_2$ is a minimal resolution of $k$, see \cite[Theorem 4]{T}. Let $\Fc = \{ \Fb(i) \}_{ i \geq 0}$ be a good filtration of  ${\Yb}_2$. We may assume that $\Fb(i)^j  = {\Yb}_2^j$ for $i \geq - m$ for all $i \geq i_0$.
We have an inclusion $\phi_{i_0} \colon \Fb(i_0) \rt \Xb = {\Yb}_2$. Then  $\phi_{i_0}$ is not a retraction, see \ref{retract}. So $g\circ \phi_{i_0} = 0$ in $K^{-,b}_f(\proj A)$ (as $\alpha$ is an AR-triangle in  $K^{-,b}_f(\proj A)$).

By \ref{extend-homotopy} for $i \geq i_0$ the null homotopy $g\circ \phi_{i}$ can be extended to a null homotopy $g\circ \phi_{i + 1}$.
Thus inductively we have defined homotopy on $\Fb(i)$ for all $i \geq i_0$. But $\Xb = \bigcup_{i \geq i_0}\Fb(i)$. It follows that $g$ is
null-homotopic which is a contradiction.

Case (2): $A$ is \emph{not} a complete intersection.

Then $\Xb = \bigcup_{i \geq 1}{\Yb}_i $ where ${\Yb}_1$ is the Koszul complex on a minimal set of generators of $\m$ and for $i \geq 2$ the DG-complex ${\Yb}_{i}$ is obtained by killing cycles of ${\Yb}_{i-1}$ in degree $-i+1$.
We note that $\Xb^n = {\Yb}_i^n$ for $n \geq -i$. By \cite[Theorem]{G2} we get that ${\Yb}_i \neq \Xb$ for all $i$.

 Let $\Fc = \{ \Fb(i) \}_{ i \geq 0}$ be a good filtration of  ${\Yb}_m$.  We may assume that $\Fb(i)^j  = {\Yb}_m^j$ for $i \geq - m$ for all $i \geq i_0$.
We have an inclusion $\phi_{i_0} \colon \Fb(i_0) \rt \Xb = {\Yb}_m$. Then clearly $\phi_{i_0}$ is not a retraction  see \ref{retract}. So $g\circ \phi_{i_0} = 0$ in $K^{-,b}_f(\proj A)$ (as $\alpha$ is an AR-triangle in  $K^{-,b}_f(\proj A)$). Then as constructed above we can extend to a homotopy from ${\Yb}_m \rt \Ub[1]$.

Now suppose the map $g \colon \Xb \rt \Ub[1]$ when restricted to ${\Yb}_i$ is null-homotopic for some $i \geq m$. Then  by \ref{extend-homotopy} we can extend the homotopy to ${\Yb}_{i+1}$. As $\Xb = \bigcup_{i \geq m}{\Yb}_i $
we have that we have extended the homotopy to $\Xb$. In particular $g$ is null-homotopic which is a contradiction.
\end{proof}

\section{Proof  of Theorem \ref{cm-gor-intro}}
In this section we  give a proof of Theorem \ref{cm-gor-intro}. We restate it for the convenience of the reader.
\begin{theorem}\label{cm-gor}
Let $(A,\m)$ be a complete Noetherian local ring. Assume \\
$K^b_f(\proj A)$ has AR-triangles.
\begin{enumerate}[\rm (1)]
  \item if $A$ is \CM  \ then $A$ is Gorenstein.
  \item if $ \dim A = 1$ then $A$ is Gorenstein.
\end{enumerate}
 then $A$ is Gorenstein.
\end{theorem}
We will need the following result:
\begin{proposition}\label{width-invar}
Let $(A,\m), (B,\n)$ be Henselian  local rings.  Let \\
$\Xb \in   K^b_f(\proj A)$ be a minimal complex with  $\width \Xb = l$. Then
\begin{enumerate}[\rm (1)]
\item
If $\Xb = 0 \rt F_l \rt \cdots \rt F_0 \rt 0$ then $\mu(\Hom(\Xb, \Xb[l])) = \rank(F_l)\rank(F_0)$.
\item
$\Hom(\Xb, \Xb[j]) = \begin{cases} \neq 0 & \text{if $l = j$} \\ = 0 & \text{for $j > l$}     \end{cases}$
\item
Let $F \colon K^b_f(\proj A) \rt K^b_f(\proj B)$ be an equivalence commuting with $[1]$. Then for any $\Xb \in   K^b_f(\proj A)$ we have $\width \Xb = \width F(\Xb)$.
\end{enumerate}
\end{proposition}
\begin{proof}
(1) We note that if $f \in \Hom(\Xb, \Xb[l])$ is null-homotopic then \\
$f \in \m \Hom_A(F_l, F_0)$. The result follows.

(2) The first assertion follows from (1). The second is obvious.

(3) We have $\Hom(\Xb, \Xb[l]) \cong \Hom(F(\Xb), F(\Xb)[l]) \neq 0$. So $\width F(\Xb) \leq \width \Xb$. By considering $F^{-1}$ we obtain the reverse inequality. The result follows.
\end{proof}

We now give
\begin{proof}[Proof of Theorem \ref{cm-gor}(1)]
If $\dim A = 0$ then result holds by  \cite[3.4]{happel-3}. So assume $d = \dim A > 0$.  Let $x_1,\ldots, x_d$ be  a system of parameters of $A$. Set $I = (x_1,..,x_d)$.  Let $\Xb$ be a minimal resolution of $A/I$.  As $K^b_f(\proj A)$ has AR-triangles we have  a Serre functor $F \colon     K^b_f(\proj A)      \rt K^b_f(\proj A)$ with
$\Hom(\Xb, \Xb) \cong \Hom(\Xb, F(\Xb))^\vee$. We note that $F$ is an equivalence. So by \ref{width-invar},  $d = \width \Xb = \width F(\Xb)$. By acyclicity lemma we get that $ F(\Xb)$ is a minimal resolution of some finite length module $M$ (upto shift). Note
$$  A/I = \Hom(\Xb, \Xb[d]) \cong \Hom(F(\Xb), F(\Xb)[d]).$$
So by  \ref{width-invar}(1) we get that $M$ is cyclic, say $M = A/J$. We note that
$$  A/I = \Hom(\Xb, \Xb) \cong \Hom(F(\Xb), F(\Xb)) = A/J.$$
So $I = J$. So we have
$$ (A/I)^\vee = \Hom(\Xb, \Xb)^\vee \cong \Hom(\Xb, F(\Xb)) = (A/I)^{\binom{d}{m}}$$
for some $m$. Counting lengths we get that $m = 0$ or $d$. Furthermore as
$$(A/I)^\vee \cong A/I $$ we get that  $A/I$ is  a Gorenstein ring. As $x_1,\ldots, x_d$ is  an $A$-regular sequence we get that $A$ is Gorenstein.
\end{proof}

Let $\Gamma(A)$ be the ideal of $\m$-torsion elements of $A $.
To prove Theorem \ref{cm-gor}(2) we need the following result:
\begin{proposition}\label{disjoint}
Let $(A,\m)$ be a one-dimensional  complete Noetherian local ring. Then there exists a parameter $x$  of $A$ with $(x) \cap  \Gamma(A) = 0$.
\end{proposition}
\begin{proof}
Let $y$ be a parameter  of $A$. By  Artin-Rees Lemma there exists $t \geq 1$ such that for $n \gg 0$ we have
$$ (y)^n \cap  \Gamma(A)  = (y)^{n-t}( (y)^t \cap  \Gamma(A).$$
We note that $(y)^{n-t}( (y)^t \cap  \Gamma(A) = 0$ for $n \gg 0$. Taking $x = y^n$ for $n$ sufficiently large yields our result.
\end{proof}
We now give
\begin{proof}[Proof of Theorem \ref{cm-gor}(2)]
By (1) it suffices to show $A$ is \CM. Suppose if possible $A$ is not \CM. So $\depth A = 0$. In particular
$\Gamma(A) \neq 0$.
By \ref{disjoint} we can choose a parameter $x$ of $A$ with $(x) \cap  \Gamma(A) = 0$.
Let
\[
\Xb \colon   0 \rt A \xrightarrow{x} A \rt 0,
\]
concentrated in degrees $0.-1$.
Note $H^0(\Xb) = A/(x)$ and $H^{-1}(\Xb) = (0 \colon x)$. Note $\ell(H^0(\Xb)) > \ell(H^{-1}(\Xb))$; see \cite[4.7.6]{BH}.
 As $K^b_f(\proj A)$ has AR-triangles we have  a Serre functor $F \colon     K^b_f(\proj A)      \rt K^b_f(\proj A)$ with
$$\Hom(\Xb, \Xb) \cong \Hom(\Xb, F(\Xb))^\vee.$$
 We note that $F$ is an equivalence. So by \ref{width-invar},  $1 = \width \Xb = \width F(\Xb)$.
 So
 \[
F(\Xb) \colon   0 \rt A^m \rt A^n \rt 0.
\]
As $\ell(H^*(F(\Xb)) < \infty$ it easily follows that $m  = n$.

As $F$ is an equivalence we get that
\[
\Hom(\Xb, \Xb[1])  \cong \Hom(F(\Xb), F(\Xb)[1]).
\]
It is easily verified that $\Hom(\Xb, \Xb[1])  = A/(x)$. By   \ref{width-invar}(1) we get that
$$\mu(\Hom(F(\Xb), F(\Xb)[1])) = m^2.$$ So $m = 1$. Let
\[
F(\Xb) \colon   0 \rt A \xrightarrow{y} A \rt 0.
\]
Note  $\Hom(F(\Xb), F(\Xb)[1])  = A/(y)$. So $(x) = (y)$. In particular  $y = \lambda x$ where $\lambda$ is a unit. Thus $F(\Xb) = \Xb[s]$ for some $s$.

As $ \Hom(\Xb, F(\Xb))^\vee  \cong  \Hom(\Xb, \Xb)  \neq 0$, we get $s = 1, 0, -1$. We will show that $s = -1$ is not possible. Furthermore we show that in both cases $s = 0, 1$ we get $\depth A > 0$. This will prove our result.

Case 1: $s = -1$. It is easily  verified that $\Hom(\Xb, \Xb[-1])  =  (0 \colon x)$. The natural map
$A \rt \Hom(\Xb, \Xb)$ has kernel $(x)$. Thus we have a natural map $\eta \colon  A/(x) \rt \Hom(\Xb, \Xb)$.
Also note that if $f \in \Hom(\Xb, \Xb)$ Then $H^0(f)\colon A/(x) \rt A/(x)$ is $A$-linear. Furthermore
$H^0(-)\circ \eta = 1_{A/(x)}$. So we have $\Hom(\Xb, \Xb) = A/(x)\oplus L$ for some finitely generated $A$-module $L$. So we have
\[
A/(x) \oplus L  = \Hom(\Xb, \Xb) \cong \Hom(\Xb, F(\Xb))^\vee = \Hom(\Xb, \Xb[-1]))^\vee  = (0 \colon x))^\vee.
\]
So we have $\ell(A/(x)) \leq \ell(0 \colon x)$ which is a contradiction.

Case 2: $s = 1$.
We have
\[
A/(x) \oplus L  = \Hom(\Xb, \Xb) \cong \Hom(\Xb, F(\Xb))^\vee = \Hom(\Xb, \Xb[1]))^\vee  = (A/x))^\vee.
\]
Computing lengths we have $L = 0$. So the natural map $\eta \colon  A/(x) \rt \Hom(\Xb, \Xb)$ is an isomorphism.
Let $t \in (0 \colon \m)$ be non-zero.
Consider the following chain map $\psi \colon \Xb \rt \Xb$ where $\psi_0 = 0$ and $\psi_{-1}$ is multiplication by $t$. As $\eta$ is an isomorphism we get that $\psi$ is multiplication by some $\alpha \in A$. So
$\eta(\alpha)- \psi$ is null-homotopic say defined by $s \colon \Xb^0 \rt \Xb^{-1}$. Then $sx = \alpha$ and $sx = \alpha - t$. So $t = 0$ a contradiction.

Case 3: $s = 0$.
In this case we have
\[
S = \Hom(\Xb, \Xb) \cong \Hom(\Xb, F(\Xb))^\vee = \Hom(\Xb, \Xb))^\vee
\]
Note as $\Xb$ is indecomposable we get that $S$ is a local  Artin $A$-algebra. As $S \cong S^\vee$ we get that $S$ is self-injective.
So $S$ has one-dimensional socle.
Let $\phi \in S$. Say $\phi_0, \phi_{-1}$ are multiplication by $u, v$ respectively. Then clearly $u$ is a unit if and only if $v$ is a unit. Furthermore if both $u, v$ are units then $\phi$ is an isomorphism.
Thus $\phi \in \rad(S)$ if and only if both $u,v \in \m$.

Let $t \in (0 \colon \m)$ be non-zero. Let $\alpha \colon \Xb \rt \Xb$ be defined by  $\alpha_0 = 0$ and $\alpha_{-1}$ is multiplication by $t$.
Also define $\beta \colon \Xb \rt \Xb$ by  $\beta_{-1} = 0$ and $\beta_{0}$ is multiplication by $t$.
Then $\alpha, \beta \in $ socle of $S$. Furthermore  it is easily verified that $\alpha \neq 0$ and $\beta \neq 0$. So $\alpha = \theta \circ \beta$ for some unit $\theta \in S$.  So
$\alpha - \theta \circ \beta $ is null-homotopic say defined by $s \colon \Xb^0 \rt \Xb^{-1}$. Note we get $sx = t$. So $t \in (x)$ which contradicts our choice of $x$.

Thus $\depth A  > 0$. So $A$ is \CM. By (1) we also get that $A$ is Gorenstein.
\end{proof}
\section{Proof of Theorem \ref{ar}}
In this section we give a proof of Theorem \ref{ar}. Throughout $(A,\m)$ is a Henselian Noetherian local ring. Set $\Kc = K^b(\proj A)$.

\s \label{setup-y} Let $\Xb \in \Kc$ be an indecomposable minimal complex. Set
\begin{align*}
  \Sc(\Xb)  =   \{ s \colon \Xb[-1] &\xrightarrow{u_s} {\Ub}_s \rt {\Kb}_s \rt  \Xb \mid  \\
   & s \ \text{ is a triangle in } \ \Kc \ \text{with ${\Ub}_s$ indecomposable and $u_s \neq 0$} \}.
\end{align*}

Our first result is
\begin{lemma}\label{non-empty}
(with hypotheses as in \ref{setup-y})  $\Sc(\Xb)$ is non-empty.
\end{lemma}
\begin{proof}
  Let $t \colon 0 \rt \Vb \rt \Pb \rt \Xb \rt 0$ be a short exact sequence in $C^b(\md(A))$ with $\Pb$ projective. It is well known that $\Pb$ is a finite direct sum of shifts of complex of the form $0 \rt A\xrightarrow{1}  A \rt 0$. Thus $t$ is a short exact sequence in $C^b(\proj A)$. So we have a triangle
  in $\Kc $
  $$ l\colon \Xb[-1] \xrightarrow{u} \Vb \rt \Pb \rt \Xb.$$
  We assert that $u \neq 0$. Otherwise $l$ is split. So $\Vb$ is a direct summand of $\Pb$. But $\Pb = 0$ in $\Kc$. So $\Vb = 0$ in $\Kc$. This forces $H^*(\Vb) = 0$. Taking long exact sequence in cohomology induced by $t$ we get $H^*(\Xb)= 0$. So by \ref{m-Lemma} we get $\Xb = 0$ in $\Kc$, a contradiction.

  Say $\Vb = \Vb(1)\oplus \Vb(2) \cdots \oplus \Vb(r)$ with $\Vb(i)$ indecomposable in $\Kc$ for all $i$. Let $u = (u_1,\ldots, u_r)$ along this decomposition. As $u\neq 0$ some $u_i \neq 0$. We note that $u_i$ induces a triangle
  $$ s\colon \Xb[-1] \xrightarrow{u_i} \Vb(i) \rt \Kb \rt \Xb.$$
  It follows that $\Sc(\Xb)$ is non-empty.
\end{proof}

\s \emph{A partial order on $\Sc(\Xb)$}. Let $s, t \in \Sc(\Xb)$. We say $s > t$ if there is a morphism of triangles $s\rt t$

\[
  \xymatrix
{
 \Xb[-1]
\ar@{->}[r]
\ar@{->}[d]^{j}
 & {\Ub}_s
\ar@{->}[r]
\ar@{->}[d]
& {\Kb}_s
\ar@{->}[r]
\ar@{->}[d]
&\Xb
\ar@{->}[d]
\\
 \Xb[-1]
\ar@{->}[r]
 & {\Ub}_t
\ar@{->}[r]
& {\Kb}_t
    \ar@{->}[r]
    &\Xb
\
 }
\]
where $j$ is the identity map.

We first prove:
\begin{lemma}\label{partial}
(with hypotheses as in \ref{setup-y}) If $s > t $ and $t > s$ then $s \cong t$.
\end{lemma}
\begin{proof}
  We compose the two morphism $s \rt t$ and $t\rt s$ to obtain a morphism $s \rt s$
  \[
  \xymatrix
{
 \Xb[-1]
\ar@{->}[r]^{u_s}
\ar@{->}[d]^{j}
 & {\Ub}_s
\ar@{->}[r]
\ar@{->}[d]^{h}
& {\Kb}_s
\ar@{->}[r]
\ar@{->}[d]
&\Xb
\ar@{->}[d]
\\
 \Xb[-1]
\ar@{->}[r]^{u_s}
 & {\Ub}_s
\ar@{->}[r]
& {\Kb}_s
    \ar@{->}[r]
    &\Xb
\
 }
\]
where $j$ is the identity map.
If $h$ is not an isomorphism then as ${\Ub}_s$ is indecomposable we get that $h $ is in the Jacobson radical of $\Rc = \Hom_\Kc({\Ub}_s, {\Ub}_s)$.
We obtain $h\circ u_s = u_s$. So we get $(1-h)\circ u_s = 0$. As $h$ is in the Jacobson radical of $\Rc$ we get that $1-h$ is invertible. So $u_s = 0$, a contradiction. So $h$ is an isomorphism. It follows that the above diagram is an isomorphism. Similarly the composite of $t\rt s$ and $s \rt t$ is an isomorphism $t \rt t$. It follows that $s\cong t$.
\end{proof}
Thus we have a well-defined partial order on $\Sc(\Xb)$. To be completely precise one has to identify elements of $\Sc(\Xb)$ upto isomorphism.

Next we show that $\Sc(\Xb)$ is directed,i.e,
\begin{lemma}\label{direct}
(with hypotheses as in \ref{setup-y}) If $s, t \in \Sc(\Xb)$ then there exists $l \in \Sc(\Xb)$ with $s > l$ and $t > l$.
\end{lemma}
\begin{proof}
  We write
  \begin{align*}
    s &\colon \Xb[-1] \xrightarrow{u_s} {\Ub}_s \rt {\Kb}_s \xrightarrow{p_s} \Xb \\
    t &\colon \Xb[-1] \xrightarrow{u_t} {\Ub}_t \rt {\Kb}_t \xrightarrow{p_t}  \Xb
  \end{align*}
  \emph{Claim}: The map $(p_s, p_t) \colon {\Kb}_s \oplus {\Kb}_t \rt \Xb$ is not a retraction. \\
  Otherwise there exists $q_s \colon \Xb \rt {\Kb}_s$ and   $q_t \colon \Xb \rt {\Kb}_t$      with
  $$ p_s \circ q_s + p_t \circ q_t = 1_{\Xb}.$$
  As $\Xb$ is indecomposable this implies that $p_s\circ q_s$ or $p_t\circ q_t$ is an automorphism on $\Xb$. This implies that $p_s$ or $p_t$ is a retraction. So $u_s = 0$ or $u_t = 0$ which is a contradiction.

  Consider the triangle
  $$ \alpha \colon \Xb[-1] \xrightarrow{v} \Vb \rt {\Kb}_s \oplus {\Kb}_t  \xrightarrow{(p_s,p_t)} \Xb.$$
  As $(p_s, p_t)$ is not a retraction we get that $v \neq 0$.
  Say $\Vb = \Vb(1)\oplus \Vb(2) \cdots \oplus \Vb(r)$ with $\Vb(i)$ indecomposable in $\Kc$ for all $i$. Let $v = (v_1,\ldots, v_r)$ along this decomposition. As $v\neq 0$ some $v_i \neq 0$. We note that $v_i$ induces a triangle
  $$ l\colon \Xb[-1] \xrightarrow{v_i} \Vb(i) \rt {\Kb}_l \rt \Xb.$$
with $l \in \Sc(\Xb)$. We have a map $\alpha  \rt  l$,
\[
  \xymatrix
{
 \Xb[-1]
\ar@{->}[r]^{v}
\ar@{->}[d]^{j}
 & {\Vb}
\ar@{->}[r]
\ar@{->}[d]^{\pi_i}
& {\Kb}_s \oplus {\Kb}_t
\ar@{->}[r]^{(p_s, p_t)}
\ar@{->}[d]
&\Xb
\ar@{->}[d]
\\
 \Xb[-1]
\ar@{->}[r]^{v_i}
 & {\Vb}(i)
\ar@{->}[r]
& {\Kb}_l
    \ar@{->}[r]
    &\Xb
\
 }
\]
Here $j$ is the identity map and $\pi$ is the obvious projection.
We also have a map $s\rt \alpha$,
\[
  \xymatrix
{
 \Xb[-1]
\ar@{->}[r]
\ar@{->}[d]^{j}
 & {\Ub}_s
\ar@{->}[r]
\ar@{->}[d]
& {\Kb}_s
\ar@{->}[r]^{p_s}
\ar@{->}[d]^{i}
&\Xb
\ar@{->}[d]
\\
 \Xb[-1]
\ar@{->}[r]
 & {\Vb}
\ar@{->}[r]
& {\Kb}_s \oplus {\Kb}_t
    \ar@{->}[r]
    &\Xb
\
 }
\]
where $j$ is the identity map and $i$ is the obvious inclusion.
Composing $s\rt \alpha$ with $\alpha \rt l$ we obtain a map $s\rt l$. So $s  > l$. Similarly $t > l$. The result follows.
\end{proof}

Thus by the \ref{direct} and \ref{partial} if there is a minimum element in $\Sc(\Xb)$ it is unique upto isomorphism of triangles.

The utility of the construction of $\Sc(\Xb)$ is that it gives a convenient way to handle left AR-triangles.
We show
\begin{theorem}
  \label{l-ar-char}
  Let $s \colon \Xb[-1] \xrightarrow{w} \Ub \rt \Kb \rt \Xb$ be a triangle in $\Kc$. Then the following conditions are equivalent:
  \begin{enumerate}[\rm(1)]
    \item $s$ is a left AR-triangle starting at $\Ub$.
    \item $s$ is the minimal element of $\Sc(\Xb)$.
  \end{enumerate}
\end{theorem}
\begin{proof}
  We first prove (2) $\implies $ (1).
  By construction of $\Sc(\Xb)$ we get that $\Xb, \Ub$ are indecomposable and $w\neq 0$.
   Let $\Db \in \Kc$ be  indecomposable and let $t \colon \Ub \rt \Db$ be a non-isomorphism. Then we have a morphism of triangles
  \[
  \xymatrix
{
 \Xb[-1]
\ar@{->}[r]^{w}
\ar@{->}[d]^{j}
 & {\Ub}
\ar@{->}[r]
\ar@{->}[d]^{t}
& {\Kb}
\ar@{->}[r]
\ar@{->}[d]
&\Xb
\ar@{->}[d]
\\
 \Xb[-1]
\ar@{->}[r]^{t\circ w}
 & {\Db}
\ar@{->}[r]
& \Pb
    \ar@{->}[r]
    &\Xb
\
 }
\]
where $j$ is the identity map and $\Pb = \cone(t\circ w)$. Let $l$ denote the lower triangle above. If $t\circ w \neq 0$ then $l \in \Sc(\Xb)$. Furthermore $s > l$. But $s$ is the minimal element in $\Sc(\Xb)$. So $l > s$. By \ref{partial} we get $s \cong l$. So $t$ is an isomorphism, a contradiction.

We now prove (1) $\implies$ (2).
Suppose if possible $s$ is not minimal element in $\Sc(\Xb)$. Then $s > l$ and $s \ncong l$ for some $l \in \Sc(\Xb)$. So we have
\[
  \xymatrix
{
 \Xb[-1]
\ar@{->}[r]^{w}
\ar@{->}[d]^{j}
 & {\Ub}
\ar@{->}[r]
\ar@{->}[d]^{t}
& {\Kb}
\ar@{->}[r]
\ar@{->}[d]
&\Xb
\ar@{->}[d]
\\
 \Xb[-1]
\ar@{->}[r]^{u_l}
 & {\Ub}_l
\ar@{->}[r]
& {\Kb}_l
    \ar@{->}[r]
    &\Xb
\
 }
\]
where $j$ is the identity map. So we get $u_l = t \circ w$. But $u_l \neq 0$ and $t$ is a non-isomorphism. This contradicts the fact that $s$ is a left AR-triangle.
\end{proof}
Next we show
\begin{theorem}
  \label{almost}
  Let  $s \colon \Xb[-1] \xrightarrow{w} \Ub \rt \Kb \rt \Xb$ be a left AR-triangle in $\Kc$. Then $H^*(\Xb)$ has finite length.
\end{theorem}
\begin{proof}
  By \ref{l-ar-char} we get that $s$ is the minimal element of $\Sc(\Xb)$. Suppose if possible $H^*(\Xb)$ does not have finite length. Let $P$ be a non-maximal prime such that
  $H^*(\Xb)_P \neq 0$. So $H^*({\Xb}_P) \neq 0$.
  Let $t \colon 0 \rt \Vb \rt \Pb \rt \Xb \rt 0$ be a short exact sequence in $C^b(\md(A))$ with $\Pb$ projective. It is well known that $\Pb$ is a finite direct sum of shifts of complex of the form $0 \rt A\xrightarrow{1}  A \rt 0$. Thus $t$ is a short exact sequence in $C^b(\proj A)$. So we have a triangle
  in $\Kc $
  $$ l\colon \Xb[-1] \xrightarrow{v} \Vb \rt \Pb \rt \Xb.$$
  By considering $t_P$ it follows that ${\Vb}_P \neq 0$ in $\Kc(\proj A_P)$. Note we have a triangle
  $$l_P \colon {\Xb}_P[-1] \xrightarrow{v_P} {\Vb}_P \rt {\Pb}_P \rt {\Xb}_P.$$
  But $\Pb$ is a contractible complex and so is ${\Pb}_P$. If $v_P$ is zero then $l_P$ splits and this forces ${\Vb}_P = 0$ a contradiction.

  Say $\Vb = \Vb(1)\oplus \Vb(2)\oplus \cdots \oplus \Vb(r)$ with $\Vb(i)$ indecomposable in $\Kc$ for all $i$. Let $v = (v_1,\ldots, v_r)$ along this decomposition. As $v_P\neq 0$ some $(v_i)_P \neq 0$. We note that $v_i$ induces a triangle in $\Sc(\Xb)$.
  $$ q \colon \Xb[-1] \xrightarrow{v_i} \Vb(i) \rt \Qb \rt \Xb.$$
  Let $\alpha \in \m \setminus P$. As $(v_i)_P \neq 0$ we get that $\alpha^n v_i \neq 0$. So for all $n \geq 1$ we have a triangle $q_n \in \Sc(\Xb)$ where
  $$ q_n \colon \Xb[-1] \xrightarrow{\alpha^n v_i} \Vb(i) \rt {\Qb}_n \rt \Xb.$$
 As $s$ is the minimal element in $\Sc(\Xb)$ we get that $q_n > \alpha$ for all $n$. So we have a morphism of triangles
 \[
  \xymatrix
{
 \Xb[-1]
\ar@{->}[r]^{\alpha^n v_i}
\ar@{->}[d]^{j}
 & {\Vb}(i)
\ar@{->}[r]
\ar@{->}[d]^{f_n}
& {\Qb}_n
\ar@{->}[r]
\ar@{->}[d]
&\Xb
\ar@{->}[d]
\\
 \Xb[-1]
\ar@{->}[r]^{w}
 & {\Ub}
\ar@{->}[r]
& {\Kb}
    \ar@{->}[r]
    &\Xb
\
 }
\]
where $j$ is the identity map. Thus $w = f_n \circ \alpha^n v_i = \alpha^n(f_n \circ v_i)$. So $w \in \m^n \Hom_{\Kc}(\Xb[-1], \Ub)$ for all $n \geq 1$. By Krull's intersection theorem we get $w = 0$ which is a contradiction.
\end{proof}

We now give a proof of Theorem \ref{ar}. We restate it in a convenient form.
\begin{theorem}
  \label{ar-main}
   In $\Kc$ we have
   \begin{enumerate}[ \rm (1)]
     \item If  $s \colon \Xb[-1] \xrightarrow{w} \Ub \rt \Kb \rt \Xb$ is a left AR-triangle then $H^*(\Xb)$, $H^*(\Ub)$ and $H^*(\Kb)$ has finite length.
     \item If $t \colon \Ub \rt \Kb \rt \Xb  \xrightarrow{h} \Ub[1] $ is a right AR-triangle then $H^*(\Xb)$, $H^*(\Ub)$ and $H^*(\Kb)$ has finite length.
   \end{enumerate}
   \end{theorem}
   \begin{proof}
(1) By \ref{almost}  we get that $H^*(\Xb)$ has finite length. So $\Hom_\Kc(\Xb, \Xb)$ is Artin local. From \ref{left-to-right} we have a right AR-triangle
$$ p \colon \Ub \rt \Kb \rt \Xb  \xrightarrow{-w[1]} \Ub[1]. $$
By \ref{ar-dual} we have a left AR-triangle
$$  \Ub^*[-1] \xrightarrow{-w^*[-1]} \Xb^* \rt \Qb \rt \Ub^*.                   $$
By \ref{almost} again we get that $H^*(\Ub^*)$ has finite length. By \ref{dual} we get that $H^*(\Ub)$ has finite length. By taking the long exact sequence of cohomology of $s$ we get that $H^*(\Kb)$ has finite length.

(2) By \ref{ar-dual} we have a left AR-triangle
$$ \Ub^*[-1] \xrightarrow{h^*} \Xb^* \rt \Pb \rt \Ub^*.$$
By (1) we get that $H^*(\Ub^*)$ and $H^*(\Xb^*)$ has finite length. By \ref{dual} we get that $H^*(\Ub), H^*(\Xb)$ has finite length. By taking the long exact sequence of cohomology of $t$ we get that $H^*(\Kb)$ has finite length.
   \end{proof}

\section{An analogue to Miyata's theorem}
In this section we prove our analogue to Miyata's theorem.
We will need the following well-known result
\begin{lemma}\label{m-Lemma}
Let $\Ga$ be a left  Noetherian ring. Let $\Xb, \Fb, \Gb \in K^{-,b}(\proj \Ga)$. We have
\begin{enumerate}[\rm (1)]
  \item If $H^i(\Xb) = 0$ for all $i \in \Z$ then $\Xb = 0$.
  \item Let $u \colon \Fb \rt \Gb$. If $H^i(u)$ is an isomorphism for all $i \in \Z$ then $u$ is an isomorphism.
\end{enumerate}
\end{lemma}

We now give
\begin{proof}[ Proof of Theorem \ref{Miyata-proj}]
Let $\theta = (\alpha, \beta) \colon \Wb \rt \Ub \oplus \Vb$ be an isomorphism.
Let $t \colon \Ub \rt \Ub \oplus \Vb \rt \Vb \xrightarrow{0} \Ub[1]$ be the obvious split triangle.

\[
  \xymatrix
{
 \Ub
\ar@{->}[r]^{u}
\ar@{->}[d]^{\alpha \circ u}
 & \Wb
\ar@{->}[r]
\ar@{->}[d]^{\theta}
& \Vb
\ar@{->}[r]
&\Ub[1]
\ar@{->}[d]^{\alpha \circ u[1]}
\\
 \Ub
\ar@{->}[r]^{i}
 & \Ub\oplus \Vb
\ar@{->}[r]^{\pi}
& \Vb
    \ar@{->}[r]^{0}
    &\Ub[1]
\
 }
\]
We note that $i \circ (\alpha\circ u) =\theta \circ u$. So there exists $h \colon \Vb \rt \Vb$ which induces a map of triangles $s \rt t$.
Taking cohomology we get
 a commutative diagram
\[
  \xymatrix
{
 H^i(\Ub)
\ar@{->}[r]
\ar@{->}[d]
 & H^i(\Wb)
\ar@{->}[r]
\ar@{->}[d]^{H^i(\theta)}
& H^i(\Vb)
\ar@{->}[r]
\ar@{->}[d]^{H^i(h)}
&H^{i+1}(\Ub)
\ar@{->}[d]^{\alpha \circ u[1]}
\\
 H^i(\Ub)
\ar@{->}[r]
 & H^i(\Ub\oplus \Vb)
\ar@{->}[r]^{H^i(\pi)}
& H^i(\Vb)
    \ar@{->}[r]^{0}
    &H^{i+1}(\Ub)
\
 }
\]
As $H^i(\theta)$ is an isomorphism and $H^i(\pi)$ is surjective a simple diagram chase shows that $H^i(h)$ is surjective. As $H^i(\Vb)$ is finitely generated $\Ga$-module and as $\Ga$ is left  Noetherian we get that $H^i(h)$ is an isomorphism. By \ref{m-Lemma} we get that $h$ is an isomorphism. It follows that $\alpha \circ u$ is also an isomorphism.
Thus the map of triangles $s \rt t$ is an isomorphism. A simple diagram chase shows that $v \colon \Vb \rt \Ub[1]$ is zero.
\end{proof}
\section{Proof of Theorem \ref{huneke-leu}}
In this section we give
\begin{proof}[Proof of Theorem \ref{huneke-leu}]
Let $u \in \Hom(\Xb, \Xb)$ be any. Let $r \in \m$ be any. Set $\Kb(n) = \cone(r^nu)$.
We have triangle
\[
\Xb    \xrightarrow{r^nu}     \Xb      \rt    \Kb(n)          \rt            \Xb[1]
\]
We note that $\Kb(n) \in \Tc$. Furthermore  $\width \Kb(n) \leq \width \Xb + 1$ and  \\ $\rank \Kb(n) \leq 2 \rank \Xb$. So by our assumption we get that $\Kb(n) \cong \Kb(m) $  in $K^b(\proj A)$ for some $n < m$. Set $v = r^nu$ and $t = r^{m-n}$. We will show $v = 0$. By \ref{hl-lemma} it follows that $H^*(\Xb)$ has finite length.
Note by \ref{cone-center} we have a  commutative diagram in $C^{b}(\md(A))$
\[
  \xymatrix
{
 0
 \ar@{->}[r]
  & \Xb
\ar@{->}[r]^{i}
\ar@{->}[d]^{t}
 & \Kb(n)
\ar@{->}[r]
\ar@{->}[d]^{\psi}
& \Xb[1]
\ar@{->}[r]
\ar@{->}[d]^{j}
&0
\\
 0
 \ar@{->}[r]
  & \Xb
\ar@{->}[r]
 & \Kb(m)
\ar@{->}[r]
& \Xb[1]
    \ar@{->}[r]
    &0
\
 }
\]
where $j$ is the identity map. We now note that $C^{b}(\md(A))$ is an abelian category. It follows that $\Kb(m)$ is the pushout of $i$ and the multiplication map on $\Xb$ by $t$.
So we have
$$ \Kb(m) \cong \frac{\Xb \oplus \Kb(n)}{\{(tx,-i(x)) \mid x \in \Xb \} }.$$
So we have an exact sequence in $C^{b}(\md(A))$,
$$ 0 \rt       \Xb          \xrightarrow{ (t,-i)}          \Xb \oplus \Kb(n)        \rt             \Kb(m)           \rt 0.$$
The above exact sequence consists of objects in $C^b(\proj A)$.  So we have a triangle in $K^b(\proj A)$
$$ \Xb          \xrightarrow{ (t,-i)}          \Xb \oplus \Kb(n)        \rt             \Kb(m)           \xrightarrow{g}  \Xb[1].$$
As $\Kb(n)    \cong       \Kb(m)$ in     $K^b(\proj A)$ it follows by our analogue of Miyata's theorem that $g = 0$ and so the above triangle is split.
Using the functor $\Hom(\Xb[1],-)$ on the above triangle we get that  the map
\[
\theta\colon \Hom(\Xb[1],\Xb[1])         \xrightarrow{(t^*[1],-i^*[1])}       \Hom(\Xb[1], \Xb[1] \oplus \Kb(n)[1])
\]
is a split injection and denote it's  left inverse by $\delta$.

We also have a triangle in $K^b(\proj A)$
$$ \Xb     \xrightarrow{v}                \Xb           \xrightarrow{i}             \Kb(n)        \rt     \Xb[1].             $$
Using the functor $\Hom(\Xb[1],-)$ on the above triangle we get an exact sequence
$$ \Hom(\Xb[1],\Xb[1]) \xrightarrow{v[1]^*} \Hom(\Xb[1],\Xb[1])   \xrightarrow{i^*[1]}    \Hom(\Xb[1],  \Kb(n)[1]).                                $$
It follows that
$$ i^*[1](v[1]) = (i^*[1]\circ v^*[1])(1_{\Xb[1]}) = 0.$$
So we get
$$ \theta(v[1]) = (tv[1],0).$$
Applying $\delta$ we get
\begin{align*}
  v[1] &=  t \delta(v[1], 0)\\
   &= t\delta( t \delta(v[1],0), 0)  \\
   &= t^2 \delta(  \delta(v[1],0), 0)
\end{align*}
Iterating we get  that $v[1]$ is infinitely divisible by $t$. So $v[1] = 0$. Thus $v = 0$.
\end{proof}


\begin{thebibliography} {99}

\bibitem{BS}
P.~Balmer and  M.~Schlichting,
 \emph{Idempotent completion of triangulated categories},
 J. Algebra 236 (2) (2001) 819--834.


\bibitem {BH}  W. Bruns and J. Herzog,
Cohen-Macaulay Rings, revised edition,
Cambridge Studies in Advanced Mathematics, 39.
Cambridge University Press, 1998.

\bibitem{G}
T.~H.~Gulliksen,
 \emph{A proof of the existence of minimal R-algebra resolutions},
  Acta Math. 120 (1968), 53--58.

\bibitem{G2}
\bysame,
 \emph{A homological characterization of local complete intersections},
  Compositio Math. 23 (1971), 251--255.

\bibitem{happel-1}
 C.~Happel,
\emph{On the derived category of a finite-dimensional algebra},
Comment. Math. Helv. 62 (1987), no. 3, 339--389.

\bibitem{Happel}
 \bysame,
\emph{Triangulated categories in the representation theory of finite-dimensional algebras},
London Mathematical Society Lecture Note Series, 119. Cambridge University Press, Cambridge, 1988.

\bibitem{happel-2}
 \bysame,
\emph{Auslander-Reiten triangles in derived categories of finite-dimensional algebras},
Proc. Amer. Math. Soc. 112 (1991), no. 3, 641--648.

\bibitem{happel-3}
 \bysame,
\emph{On Gorenstein algebras},
 Representation theory of finite groups and finite-dimensional algebras (Bielefeld, 1991),
389--404,
Progr. Math., 95, Birkhäuser, Basel, 1991.

\bibitem{HL}
C.~Huneke and G.~J.~Leuschke,
\emph{Two theorems about maximal Cohen-Macaulay modules},
Math. Ann. 324 (2002), no. 2, 391--404.

\bibitem {LC}
J.~Le and X.~Chen,
\emph{Karoubianness of a triangulated category},
J. Algebra 310 (2007),  452--457.

\bibitem {LW}
G.~J.~Leuschke and R.~Wiegand,
\emph{Cohen-Macaulay representations},
Mathematical Surveys and Monographs, 181. American Mathematical Society, Providence, RI, 2012.

 \bibitem{Mat}
H.~Matsumura, \emph{Commutative ring theory}, second ed., Cambridge
  Studies in Advanced Mathematics, vol.~8, Cambridge University Press,
  Cambridge, 1989.

\bibitem{Mi}
T.~Miyata,
Note on direct summands of modules.
J. Math. Kyoto Univ. 7 (1967), 65–69.

 \bibitem{NK}
A.~Neeman,
\emph{The chromatic tower for $D(R$}),
With an appendix by Marcel B\"{o}kstedt.
Topology 31 (1992), no. 3, 519--532.


 \bibitem{N}
 A.~Neeman,
 \emph{Triangulated categories},
  Annals of Mathematics Studies, 148. Princeton University Press, Princeton, NJ, 2001.

 \bibitem{RV}
   I.~Reiten and M.~Van den Bergh,
   \emph{ Noetherian hereditary abelian categories satisfying Serre duality},
    J. Amer. Math. Soc. 15 (2002), no. 2, 295--366.

  \bibitem{S}
    C.~Schoeller,
    \emph{Homologie des anneaux locaux noeth\'{e}riens},
     C. R. Acad. Sci. Paris S\'{e}r. A-B 265 (1967), A768--A771.

\bibitem{T}
J.~Tate,
\emph{Homology of Noetherian rings and local rings},
Illinois J. Math. 1 (1957), 14--27.

\bibitem{Y}
Y.~Yoshino,
\emph{Cohen-Macaulay modules over Cohen-Macaulay rings},
London Mathematical Society Lecture Note Series, 146. Cambridge University Press, Cambridge, 1990.


\end{thebibliography}
 \end{document}